\pdfoutput=1
\documentclass{amsart}
\usepackage{
 amssymb,
 enumerate,
 fullpage,
 mathtools,
 tikz-cd}
\usepackage[pagebackref]{hyperref} % likes to be last

\renewcommand{\AA}{\mathbf{A}}
\newcommand{\DD}{\mathbf{D}}
\newcommand{\CC}{\mathbf{C}}
\newcommand{\GG}{\mathbf{G}}
\newcommand{\PP}{\mathbf{P}}
\newcommand{\QQ}{\mathbf{Q}}
\newcommand{\ZZ}{\mathbf{Z}}

\newcommand{\fN}{\mathfrak{N}}
\newcommand{\fp}{\mathfrak{p}}
\newcommand{\fpb}{\bar{\mathfrak{p}}}
\newcommand{\fP}{\mathfrak{P}}

\newcommand{\cD}{\mathcal{D}}
\newcommand{\cE}{\mathcal{E}}
\newcommand{\cF}{\mathcal{F}}
\newcommand{\cK}{\mathcal{K}}
\newcommand{\cL}{\mathcal{L}}
\newcommand{\cM}{\mathcal{M}}
\newcommand{\cO}{\mathcal{O}}
\newcommand{\cR}{\mathcal{R}}
\newcommand{\cV}{\mathcal{V}}
\newcommand{\cW}{\mathcal{W}}

\newcommand\ac{\mathrm{ac}}
\newcommand{\ab}{\mathrm{ab}}
\newcommand{\et}{\text{\textup{\'et}}}
\newcommand{\f}{\mathrm{f}}
\newcommand{\cris}{\mathrm{cris}}
\newcommand{\rig}{\mathrm{rig}}

\DeclareMathOperator{\Gal}{Gal}
\DeclareMathOperator{\GL}{GL}
\DeclareMathOperator{\Ind}{Ind}
\DeclareMathOperator{\loc}{loc}
\DeclareMathOperator{\Res}{Res}
\DeclareMathOperator{\CH}{CH}
\DeclareMathOperator{\mom}{mom}
\DeclareMathOperator{\Sym}{Sym}
\DeclareMathOperator{\TSym}{TSym}
\DeclareMathOperator{\Hom}{Hom}
\DeclareMathOperator{\Sh}{Sh}

\newcommand{\Qb}{\overline{\QQ}}
\newcommand{\Qp}{\QQ_p}
\newcommand{\Zp}{\ZZ_p}
\newcommand{\Zpnh}{\widehat{\ZZ}_p^{\mathrm{nr}}}
\newcommand{\Qpnh}{\widehat{\QQ}_p^{\mathrm{nr}}}

\newcommand{\wM}{\widetilde{\cM}}
\newcommand{\wD}{\widetilde{\cD}}
\newcommand{\OK}{\cO_{\cK}}
\newcommand{\Dcris}{\DD_{\cris}}
\newcommand{\Drig}{\DD^{\dag}_{\rig}}

\newcommand{\BDP}{L_{\fp}^{\mathrm{BDP}}}
\newcommand{\htimes}{\mathop{\hat\otimes}}
\newcommand{\veps}{\varepsilon}
\renewcommand{\ge}{\geqslant}
\renewcommand{\le}{\leqslant}
\newcommand{\into}{\hookrightarrow}

\newcommand{\tp}[1]{\texorpdfstring{$#1$}{#1}}

\newcommand{\mtwo}[4]{
 \left(
 \begin{smallmatrix}#1&#2\\#3&#4
 \end{smallmatrix}
 \right)
}

\makeatletter
\@addtoreset{equation}{subsection}

\makeatother

\newtheorem{definition}[subsubsection]{\bf Definition}
\newtheorem{corollary}[subsubsection]{\bf Corollary}
\newtheorem{lemma}[subsubsection]{\bf Lemma}
\newtheorem{theorem}[subsubsection]{\bf Theorem}
\newtheorem{proposition}[subsubsection]{\bf Proposition}
\newtheorem{notation}[subsubsection]{\bf Notation}

\newtheorem{maintheorem}{Theorem}

\theoremstyle{definition}
\newtheorem{rmk}[subsubsection]{\it Remark}

\author{Dimitar Jetchev}
\address{Dimitar Jetchev, Ecole Polytechnique F\'ed\'erale de Lausanne, EPFL SB MATHGEOM GR-JET, Switzerland}
\email{dimitar.jetchev@epfl.ch}

\author{David Loeffler}
\address{David Loeffler, Mathematics Institute\\
 University of Warwick\\
 Coventry CV4 7AL, UK.}
\email{d.a.loeffler@warwick.ac.uk}
\urladdr{\href{http://orcid.org/0000-0001-9069-1877}{0000-0001-9069-1877}}

\author{Sarah Livia Zerbes}
\address{Sarah Livia Zerbes, Department of Mathematics\\
 University College London\\
 London WC1E 6BT, UK.}
\email{s.zerbes@ucl.ac.uk}
\urladdr{\href{http://orcid.org/0000-0001-8650-9622}{0000-0001-8650-9622}}

\thanks{Supported by the following grants: Swiss National Science Foundation grant PP00P2-144658 (Jetchev); Royal Society University Research Fellowship ``$L$-functions and Iwasawa theory'' (Loeffler); ERC Consolidator Grant ``Euler systems and the Birch--Swinnerton-Dyer conjecture'' (Zerbes).}

\title{Heegner Points in Coleman Families}

\begin{document}

%\pgfversion

\begin{abstract}
 We construct two-parameter analytic families of Galois cohomology classes interpolating the \'etale Abel--Jacobi images of generalised Heegner cycles, with both the modular form and Gr\"ossencharacter varying in $p$-adic families.
\end{abstract}

\maketitle

% Sections
\section{Introduction}

 \subsection{Statement of the theorems}

 Let $f$ be a cuspidal modular newform of level $\Gamma_1(N)$ and weight $\ge 2$, and $\cK$ an imaginary quadratic field in which all primes dividing $N$ are split (the ``Heegner condition''). The landmark article of Bertolini--Darmon--Prasanna \cite{BDP} defines a family of algebraic cycles -- \emph{generalised Heegner cycles} -- associated to twists of $f$ by suitable algebraic Gr\"ossencharacters of $\cK$. This construction includes, as special cases, the familiar Heegner points of Gross--Zagier \cite{gross-zagier} and Kolyvagin \cite{kolyvagin:euler_systems}
 and the Heegner cycles studied by Schoen \cite{schoen1986}, Nekov\'{a}\v{r} \cite{nekovar:kolyvagin}, and others. We refer to the pairs $(f, \chi)$ to which this construction applies as \emph{Heegner pairs}; the conditions on $(f, \chi)$ are recalled in \S\ref{sect:heegnerpairs} below.

 Via the \'etale Abel--Jacobi map, the generalised Heegner cycle for a Heegner pair $(f, \chi)$ gives rise to a class in the cohomology of the conjugate-self-dual Galois representation $V_p(f)^*\otimes \chi$, where $V_p(f)$ is Deligne's $p$-adic representation associated to $f$, and $V_p(f)^* = \Hom(V_p(f), \Qb_p)$. If we assume $p \nmid N$, then these Galois representations naturally vary in 2-dimensional $p$-adic families: one dimension comes from varying $\chi$, and a second dimension comes from varying the modular form $f$ in a Hida or Coleman family. So it is natural to ask whether these \'etale generalised Heegner classes interpolate $p$-adically in these families.

 More precisely, we define a \emph{$p$-stabilised Heegner pair} (see \S \ref{sect:coleman}) to be a triple $(f, \alpha, \chi)$, where $(f, \chi)$ is a Heegner pair, and $\alpha$ is a root of the Hecke polynomial of $f$ at $p$. We introduce in Definition \ref{def:Keigenvar} a 2-dimensional $p$-adic rigid space $\cE_{\cK}(\fN)$, which is a fibre bundle over the level $N$ Coleman--Mazur eigencurve $\cE(N)$; pairs $(f, \alpha)$ determine points on $\cE(N)$, and the covering $\cE_{\cK}(\fN)$ parametrises the additional choice of the character $\chi$. By construction, $p$-stabilised Heegner pairs $(f, \alpha, \chi)$ give rise to points of $\cE_{\cK}(\fN)$, lying over the point $(f, \alpha)$ of $\cE(N)$. We shall show that Heegner classes can be interpolated over neighbourhoods of classical points in $\cE_{\cK}(\fN)$.

 We choose a point of $\cE(N)$ corresponding to a $p$-stabilised eigenform $(f_0, \alpha_0)$ which is \emph{noble} in the sense of \cite{hansen}. Then we can find a neighbourhood $V \ni (f_0, \alpha_0)$ in $\cE(N)$ which is smooth, and \'etale over weight space, corresponding to a Coleman family $\cF$ passing through $(f_0, \alpha_0)$; and a free rank 2 $\cO(V)$-module $M_{V}(\cF)$ with an action of $\Gal(\Qb/\QQ)$ that interpolates the Galois representations $V_p(f)$ for all classical points $(f, \alpha)$ of $V$ (including $(f_0, \alpha_0)$). Let $\tilde V$ be the preimage of $V$ in $\cE_{\cK}(\fN)$, and $\kappa_{\tilde V}$ the universal family of characters over $\tilde V$.

 \begin{maintheorem}
  If $V$ is small enough, there exists a rigid-analytic family of cohomology classes
  \[ z_{\cF} \in H^1(\cK, M_{V}(\cF)^* \otimes \kappa_{\tilde V}), \]
  the ``big Heegner class'', whose specialisation at any point of $\tilde V$ corresponding to a $p$-stabilised Heegner pair $(f, \alpha, \chi)$ is an explicit scalar multiple depending on $\alpha$, of the generalised Heegner class associated to $(f, \chi)$.
 \end{maintheorem}

 Note that we make no assumption on the local behaviour of $p$ in $\cK$: it can be split, inert, or ramified. The proof of Theorem A relies heavily on the techniques developed by two of us in the paper \cite{loeffler-zerbes:coleman} to interpolate Euler systems for $\GL_2 \times \GL_2$ in Coleman families. As in \emph{op.cit.}, the main tool is the theory of \emph{overconvergent \'etale cohomology} introduced by Andreatta--Iovita--Stevens \cite{andreatta-iovita-stevens:overconvES}.

 In \S \ref{sect:selmer}, we study the local properties of the Heegner class $z_{\cF}$, and show that it is naturally a section of a ``Selmer sheaf'' over $\tilde{V}$, defined using Pottharst's theory of trianguline Selmer complexes. It would be interesting to attempt to formulate (and maybe even to prove) an Iwasawa main conjecture in this context; we have not attempted to do this here, for reasons of space, but it will be treated in forthcoming work by members of our research groups.

 In the final two sections of the paper, we impose the additional assumption that $p$ splits in $\cK$. In this case, Bertolini, Darmon and Prasanna have defined a ``square root $p$-adic $L$-function'' $\BDP(f)$, interpolating square roots of central $L$-values $L(f, \chi^{-1}, 1)$ for varying $\chi$. We show in Theorem \ref{thm:B2} that the BDP $L$-functions for varying $f$ interpolate to a 2-parameter $p$-adic $L$-function $\BDP(\fN)$ on $\cE_\cK(\fN)$. Our final main result relates this ``analytic'' $p$-adic $L$-function to the Heegner class in Galois cohomology from Theorem A:

 \begin{maintheorem}
  If $V$ is sufficiently small, then the image of the class $z_{\cF}$ under the Perrin-Riou regulator map is the restriction to $\tilde V$ of the $p$-adic $L$-function $\BDP(\fN)$.
 \end{maintheorem}

 (In contrast to Theorem A, the assumption that $p$ split in $\cK$ is essential for our methods here. Andreatta and Iovita \cite{AI-katz} have recently constructed an analogue of the BDP $L$-function of an individual cuspform $f$ when $p$ is inert or ramified in $\cK$, and it would be interesting to investigate whether this can be extended to allow $f$ varying in Coleman families; we understand that this question will be treated in forthcoming work of their research groups.)

 \subsection{Relation to earlier work} Theorem A extends earlier work of a number of authors. In particular, if $\cF$ is an ordinary Coleman family (i.e.~a Hida family), then Howard \cite{howard:variation} has constructed a family of cohomology classes $z_{\cF}$ interpolating the Heegner points at \emph{weight 2} specialisations of $\cF$. It is, however, not clear from Howard's construction whether the specialisations of the family $z_{\cF}$ at higher-weight classical points $(f, \chi)$ coincide with the generalised Heegner cycles, which is the content of Theorem A in this case.

 This compatibility is known if $\chi$ has $\infty$-type $(\tfrac{k}{2}, \tfrac{k}{2})$ and $p$ is split in $\cK$, by work of Castella \cite{castella:padicvar}. Castella has informed us that his method could also be made to work for $\chi$ of more general $\infty$-type; but the restriction to $p$ split is fundamental, since his strategy involves \emph{deducing} Theorem A from a version of Theorem B for Hida families proved in \cite{castellahsieh}, rather than the other way around as in our approach. A direct proof of Theorem A for ordinary families, without assuming $p$ split, can be extracted from recent work of Disegni \cite{disegni}; Disegni's methods are rather closer in spirit to those of the present paper.

 The case of non-ordinary Coleman families does not seem to have received so much attention hitherto. As far as we are aware, the only result in this direction prior to the present work is in unpublished work of Kobayashi \cite{kobayashi18} , which gives a one-parameter family of cohomology classes interpolating the Heegner classes for $(f, \chi)$, for $f$ a fixed, possibly non-ordinary eigenform and varying $\chi$ (assuming $p$ is split in $\cK$ and $f$ has trivial character). This result can be recovered from Theorem A by specialising our family $z_{\cF}$ to the fibre of $\tilde V$ over the point of $V$ corresponding to $f_\alpha$. \medskip

 \noindent\emph{Note}: Shortly after the present paper was first uploaded to the arXiv in June 2019, two other papers appeared focussing on the same problem. The preprint \cite{BL19} of B\"uy\"ukboduk and Lei extends Kobayashi's work to Coleman families, thus giving an alternative proof of Theorem A in the split-prime case. Their proof is significantly different from ours: as with Castella's work in the ordinary case, they first prove an analogue of Theorem B and then deduce Theorem A from it.

 The paper \cite{ota20} proves the same result as our Theorem A in the case of $p$ split in $\cK$ and $\cF$ a Hida family, assumptions similar to those of \cite{castella:padicvar} although the result is stronger than that of \emph{op.cit.} (weaking the assumptions on $\cF$ and allowing all $\infty$-types of $\chi$). Ota's methods are rather different from Castella's, and appear to be somewhat closer to those of \cite{disegni} and the present paper.

 \subsection{Potential generalisations} Our result raises several natural questions which we hope will be addressed in future work. Firstly, although we have imposed a strong Heegner hypothesis which allows us to work with classical modular curves, the construction should extend straightforwardly to Heegner points arising from quaternionic Shimura curves attached to (indefinite) quaternion division algebras over $\QQ$. It should also be possible to extend the construction to Shimura curves over totally real fields; Disegni's results for ordinary families already apply in this generality, and we hope that it may be possible to extend our results to treat the non-ordinary case.

 One can also consider whether the cohomology classes $z_{\cF}$ glue together for different $\cF$'s. The space $\tilde V$ is a small local piece of an eigenvariety for the quasi-split unitary group $\operatorname{GU}(1, 1)$ associated to $\cK$, and it is natural to ask whether the family $z_{\cF}$ can be ``globalised'' over the whole eigenvariety. A result of this kind for Kato's Euler system has been announced in a preprint of Hansen \cite{hansen}, and it would be interesting to generalise this to the Heegner setting.

 The original motivation for this project was to study the behaviour of Heegner classes in the neighbourhood of a critical-slope Eisenstein series $f$. Here the eigenspace in classical \'etale cohomology associated to $f$ is 1-dimensional, and the projection of the Heegner class to this eigenspace is trivially zero; but the eigenspace in overconvergent \'etale cohomology is larger, and we hope that it may be possible to obtain interesting cohomology classes by projecting our families $z_{\cF}$ to these ``shadow'' Eisenstein eigenspaces.

 \subsection{Acknowledgements} This project was begun while the authors were participants in the semester ``Euler Systems and Special Values of L-functions'' at the Bernoulli Centre, EPFL, Switzerland. We are grateful to the Centre Bernoulli Centre for its hospitality.

\section{Notation and conventions}

 \subsection{Gr\"ossencharacters}
  \label{sect:grossen}

  We recall the setting in which Heegner points and Heegner cycles are considered, following \cite[\S 4.1]{BDP}. We shall fix an integer $N$, which will be the prime-to-$p$ part of the level of the modular forms we shall consider.

  Let $\cK$ be an imaginary quadratic field satisfying the classical Heegner hypothesis, that all primes $q \mid N$ are split\footnote{One could also allow some primes $q\mid N$ to ramify in $\cK$ subject to assumptions on the epsilon-factors, as in \cite{BDP}, but for simplicity we shall stick to the above setting.} in $\cK$. This hypothesis implies that there exist ideals $\fN \triangleleft \OK$ such that $\OK / \fN \cong \ZZ /N\ZZ$. We shall choose such an ideal $\fN$. For any integer $c$ coprime to $N$, there is an isomorphism $\cO_c / (\fN \cap \cO_c) \cong \ZZ /N\ZZ$, where $\cO_c$ is the $\OK$-order of conductor $c$. So we can regard any Dirichlet character $\veps$ modulo $N$ as a character of $(\cO_c / \fN \cap \cO_c)^\times$, and hence of the profinite completion $\widehat{\cO}_c^\times$.

  \begin{definition}
   An \emph{algebraic Gr\"ossencharacter of $\cK$ of finite type $(c, \fN, \veps)$ and $\infty$-type\footnote{Our conventions for $\infty$-types are consistent with \cite[\S 4.1]{BDP}, but note that some other references would define this to be $\infty$-type $(-a, -b)$.} $(a, b)$} is a continuous homomorphism $\chi: \AA_{\cK, \f}^\times \to \Qb^\times$ satisfying $\chi(x) = x^a \bar{x}^b$ for $x \in \cK^\times$, whose conductor is divisible by $c$, and whose restriction to $\widehat{\cO}_c^\times$ is $\veps^{-1}$.
  \end{definition}

  The conditions imply that the conductor of $\chi$ is exactly $c \fN_{\veps}$, where $\fN_{\veps}$ is the unique factor of $\fN$ whose norm is the conductor of $\veps$.

  \begin{rmk}
   Gr\"ossencharacters of finite type $(c, \fN, \veps)$ and $\infty$-type $(a, b)$ exist if and only if $\veps(u \bmod \fN) = u^{-a} \bar{u}^{-b}$ for all $u \in \cO_c^\times$. If $\cO_c^\times = \{\pm 1\}$ (which is true in virtually all cases, i.e. all except the case when $c = 1$ and $\cK$ is either $\QQ(\sqrt{-1})$ or $\QQ(\sqrt{-3})$), this reduces to the condition $\veps(-1) = (-1)^{a + b}$.
  \end{rmk}

 \subsection{Heegner pairs}
  \label{sect:heegnerpairs}

  \begin{definition}
   A \emph{Heegner pair} of finite type $(c, \fN, \veps)$ and $\infty$-type $(a,b)$ is a pair $(f, \chi)$, where
   \begin{itemize}
    \item $\chi$ is an algebraic Gr\"ossencharacter of finite type $(c, \fN, \veps)$ and $\infty$-type $(a, b)$, with $a,b \ge 0$;
    \item $f$ is a normalised cuspidal modular newform of level $\Gamma_1(N)$, character $\veps$, and weight $a + b + 2$.
   \end{itemize}
   We say $(f, \chi)$ has \emph{finite type $(c, \fN)$} if it has finite type $(c, \fN, \veps)$ for some $\veps$ modulo $N$ (not necessarily primitive).
  \end{definition}

  If $(f, \chi)$ is a Heegner pair, then $L(f, \chi^{-1}, s)$ has a functional equation centred at $s = 1$, and the sign in the functional equation is $-1$, so $L(f, \chi^{-1}, 1) = 0$.
  \begin{rmk}
   Let $\mathbf{N}$ denote the Gr\"ossencharacter of $\infty$-type $(1, 1)$ such that for all primes $\mathfrak{q}$ of $\cK$, $\mathbf{N}$ maps a uniformizer at $\mathfrak{q}$ to $\#\left( \OK / \mathfrak{q}\right)$. Then $(f, \chi)$ is a Heegner pair precisely when the character $\chi \cdot \mathbf{N}$ lies in the set $\Sigma_{\mathrm{cc}}^{(1)}(\fN)$ associated to $f$, as defined in \S 4.1 of \cite{BDP}. We find it more convenient to work with $\chi$ rather than with $\chi \cdot \mathbf{N}$, which has the effect of shifting the centre of the functional equation from $s = 0$ to $s = 1$.
  \end{rmk}

  \subsection{Galois representations}
   \label{sect:galreps}

   We briefly fix notations for Galois representations. Let $p$ be a prime not dividing $N$, and fix an embedding $\Qb \into \Qb_p$. For simplicity, we suppose $p \ne 2$. We shall also suppose for the remainder of this paper that $N \ge 4$; the case of $N = 1, 2, 3$ can be dealt with by introducing auxiliary level structure in the usual way, but we leave the details to the interested reader.

   \subsubsection{Modular forms} Let $f$ be any cuspidal Hecke eigenform of level $\Gamma_1(N)$ and weight $k + 2 \ge 2$. Via our fixed embedding $\Qb \into \Qb_p$, we can consider the Hecke eigenvalues $a_\ell(f)$ as elements of $\Qb_p$. We write $V_p(f)$ for Deligne's Galois representation associated to $f$, which is the unique isomorphism class of two-dimensional $\Qb_p$-linear representations of $\Gal(\Qb/\QQ)$ on which the trace of geometric Frobenius at a prime $\ell \nmid Np$ is $a_\ell(f)$. Note that the Hodge--Tate weights of $V_p(f)$ are $\{0, -1-k\}$ (where the Hodge--Tate weight of the cyclotomic character is taken to be $+1$).

   \subsubsection{Gr\"ossencharacters as Galois characters} Recall that the Artin map identifies $\Gal(\cK^{\ab}/\cK)$ with the profinite completion of $\AA_{\cK, \f}^\times / \cK^\times$. (We normalise this map so that uniformisers map to \emph{geometric} Frobenius elements.)

   Our choice of embeddings $\cK\subset \Qb \into \Qb_p$ gives us two canonical characters $\sigma, \bar\sigma: (\cK\otimes\Qp)^\times \to \Qb_p^\times$. If $\chi: \AA_{\cK}^\times / \cK^\times \to \CC^\times$ is an algebraic Gr\"ossencharacter of $\infty$-type $(a,b)$ (see footnote above), then the map $\AA_{\cK, \f}^\times \to \Qb_p^\times$ defined by $x \mapsto \chi(x) \sigma(x_p)^{-a} \bar\sigma(x_p)^{-b}$ is trivial on $\cK^\times$ and hence can be regarded as a Galois character. It can be characterised as the unique character unramified outside $p\mathfrak{f}_\chi$ that maps a geometric Frobenius at $\mathfrak{q}$ to $\chi(\varpi_\mathfrak{q})$ for all primes $\mathfrak{q} \nmid p \mathfrak{f}_\chi$, where $\mathfrak{f}_\chi$ is the conductor of $\chi$ and $\varpi_\mathfrak{q}$ is a uniformiser at $\mathfrak{q}$. Note that this construction sends the norm character $\mathbf{N}$ to the inverse cyclotomic character, and the Hodge--Tate weights of $\chi$ at the two embeddings $\cK\into \Qb_p$ are $(-a, -b)$.

   \begin{rmk}
    Our conventions are chosen such that $(f, \chi)$ is a Heegner pair, then the representation $V = V_p(f)^* \otimes \chi$ of $\Gal(\Qb/\cK)$ satisfies $V^\tau \cong V^*(1)$, where $\tau$ is the non-trivial element of $\Gal(\cK/\QQ)$.
   \end{rmk}

   \begin{definition}
    \label{def:fields}
    We consider the following abelian extensions of $\cK$, corresponding via the Artin map to open compact subgroups of $\AA_{\cK, \f}^\times$:
    \begin{itemize}
     \item $F$ denotes the extension corresponding to $(1 + \fN\widehat{\cO}_{\cK})^\times$ (the ray class field modulo $\fN$);
     \item For $m \ge 0$, $\cK_m$ is the extension corresponding to $\widehat{\cO}_{p^m}^\times$ (the ring class field modulo $p^m$);
     \item $F_m$ denotes the extension corresponding to $(1 + \fN\widehat{\cO}_{\cK})^\times \cap \widehat{\cO}_{p^m}^\times$.
    \end{itemize}
   \end{definition}

   Since $N \ge 4$, we have $\OK^\times \cap (1 + \fN) = \{1\}$, so the Artin map restricts to an isomorphism from $(1 + \fN\widehat{\cO}_{\cK})^\times$ to $\Gal(\cK^{\ab} / F)$, where $F$ is the ray class field modulo $\fN$. We thus have a character $\sigma: \Gal(\cK^{\ab} / F) \to \Qb_p^\times$ given by $x \mapsto \sigma(x_p)^{-1}$ on $1 + \fN\widehat{\cO}_{\cK}$, and similarly for $\bar{\sigma}$. If $\chi$ is a Groessencharacter of finite type $(p^m, \fN, \veps)$ and $\infty$-type $(a, b)$, then the Galois character associated to $\chi$ restricts to $\sigma^a \bar\sigma^b$ on $\Gal(\cK^{\ab} / F_m)$. On the slightly larger group $\Gal(\cK^{\ab} / \cK_m)$, it is given by the character of $\widehat{\cO}_{p^m}^\times$ mapping $x$ to $\veps(x \bmod \fN)^{-1} \sigma(x_p)^{-a} \bar\sigma(x_p)^{-b}$; we abuse notation slightly by denoting this Galois character by $\veps \sigma^a \bar\sigma^b$. Thus the Groessencharacters of $\infty$-type $(a,b)$ and finite type $(p^n, \fN, \veps)$ for $n \le m$ correspond to extensions of $\veps \sigma^a \bar\sigma^b$ from $\Gal(\cK^{\ab} / \cK_m)$ to $\Gal(\cK^{\ab} / \cK)$.

 \subsection{Coleman families}
  \label{sect:coleman}

  Recall that the \emph{Hecke polynomial} at $p$ of a newform $f \in S_{k+2}(\Gamma_1(N), \veps)$ is the quadratic polynomial $X^2 - a_p(f) X + p^{k+1} \veps(p)$. Each root $\alpha$ of the Hecke polynomial corresponds to a normalised eigenform $f_\alpha$ of level $\Gamma_1(N) \cap \Gamma_0(p)$ with the same Hecke eigenvalues as $f$ away from $p$, and $U_p$-eigenvalue $\alpha$; we refer to these forms $f_\alpha$ as \emph{$p$-stabilisations} of $f$.

  \begin{definition}
   A \emph{$p$-stabilised Heegner pair} $(f, \alpha, \chi)$ of tame level $\fN$ consists of a Heegner pair $(f, \chi)$ of finite type $(p^s, \fN)$, for some $s \ge 0$, and a choice of root $\alpha$ of its Hecke polynomial.
  \end{definition}

  These are naturally parametrised by a rigid-analytic space, as we now explain. Let $L$ be a finite extension of $\Qp$ containing the values of $\veps$. We write $\cW$ for the 1-dimensional rigid-analytic group variety over $L$ parametrising continuous $p$-adic characters of $\Zp^\times$; we consider $\ZZ$ as a subset of $\cW(L)$ by mapping $k$ to the character $x \mapsto x^k$.

  There is a rigid-analytic space $\cE(N, \veps) \to \cW$, the $N$-new, cuspidal, character $\veps$ part of the \emph{Coleman--Mazur--Buzzard eigencurve} \cite{coleman-mazur, buzzard}, and $p$-stabilised eigenforms $f_\alpha$ as above are naturally points of $\cE(N, \veps)$. (Here, as in \cite{loeffler-zerbes:coleman}, our conventions are such that the fibre over $k \in \ZZ$ corresponds to overconvergent eigenforms of weight $k + 2$.) There is also a rigid-analytic space $\cW_\cK(\fN, \veps)$ parametrising continuous $p$-adic characters of $\AA_{\cK, \f}^\times /\cK^\times$ whose restriction to $(\widehat{\cO}_{\cK}^{(p)})^\times$ is $\veps^{-1}$. Mapping a character of $\AA_{\cK, \f}^\times$ to the inverse of its restriction to $\Zp^\times \subseteq \cO_{\cK,p}^\times$ defines a morphism $\cW_\cK(\fN, \veps) \to \cW$.

  \begin{definition}
   \label{def:Keigenvar}
   With the above notations, we define $\cE_{\cK}(\fN, \veps)$ to be the fibre product
   \[
   \cE_{\cK}(\fN, \veps) \coloneqq \cW_\cK(\fN, \veps) \mathop{\times}_{\cW} \cE(N, \veps).
   \]
   We define $\cE_{\cK}(\fN) = \bigsqcup_{\veps} \cE_{\cK}(\fN, \veps)$, and similarly $\cE(N)$ and $\cW_{\cK}(\fN)$.
  \end{definition}

  This will be the parameter space for our $p$-adic families. By construction, any $p$-stabilised Heegner pair of tame level $\fN$ gives a point on $\cE_{\cK}(\fN)$, which lies above $k \in \ZZ \subset \cW$ if $f$ has weight $k + 2$.

  \begin{rmk}[Notes on the space $\cE_{\cK}(\fN)$] \
   \begin{enumerate}[(i)]
    \item We can interpret $\cE_{\cK}(\fN)$ as an eigenvariety for the quasi-split unitary group
    \[ (\GL_2 \times \Res_{\cK/\QQ} \GG_m) / \GG_m \cong \operatorname{GU}(1, 1).\]

    \item The space $\cE_{\cK}(\fN)$ is a torsor over $\cE(N)$ for the rigid-analytic group variety $\cW_\cK^\ac$ parametrising characters of $\AA_{\cK, \f}^\times / ( \cK^\times\cdot  \widehat{\cO}_{p^\infty}^\times ) \cong \Gal(\cK_\infty /\cK)$.

    \item If $\veps_0$ is the trivial character mod $N$, then the images of both $\cW_{\cK}(\fN, \veps_0)$ and $\cE(N, \veps_0)$ in $\cW$ are contained in the subset $\cW_+$ of the components of $\cW$ parametrising characters with $\kappa(-1) = 1$. Over this subset, one can choose a character $\Theta: \Zp^\times \to \cO(\cW_+)^\times$ satisfying $\Theta^2 = \kappa^{\mathrm{univ}}$, where $\kappa^{\mathrm{univ}}$ is the canonical character. The composite of $\Theta$ with the norm map $\AA_{\cK, \f}^\times / \cK^\times \xrightarrow{\operatorname{Nm}_{\cK/\QQ}} \AA_{\QQ, \f}^\times / \QQ^\times_{>0} \cong \widehat\ZZ^\times \xrightarrow{\Theta} \cO(\cW_+)$ then gives a splitting of the map $\cW_{\cK}(N, \veps_0) \to \cW_+$, and hence an identification
    \[  \cE_{\cK}(\fN, \veps_0) \cong \cE(N, \veps_0) \times \cW_{\cK}^{\mathrm{ac}}. \]
    This is the approach adopted in \cite{howard:variation} and many other subsequent works such as \cite{castella:padicvar}. However, the choice of the square-root character $\Theta$ is non-canonical, and it is awkward to handle non-trivial $\veps$ by this approach, particularly when $\veps(-1) = -1$. So we prefer to use the space $\cE_{\cK}(\fN)$ instead.
   \end{enumerate}
  \end{rmk}

\section{Generalised Heegner cycles}

 \subsection{Groups and embeddings}
  \label{sect:groups}

  If $\tau \in \cK - \QQ$, there is a unique embedding of $\QQ$-algebras $\iota: \cK \into \operatorname{Mat}_{2 \times 2}(\QQ)$ such that $\iota(\cK)$ fixes the line in $\cK^2$ spanned by $\left(\begin{smallmatrix} \tau \\ 1 \end{smallmatrix}\right)$ and acts on the corresponding line by the natural scalar multiplication.

  \begin{rmk}
   Concretely, $\iota$ maps $a + b\tau \in \cK$ to the matrix $\begin{pmatrix} a + (\tau + \bar{\tau}) & -b\tau\bar{\tau} \\ b & a \end{pmatrix} \in M_{2\times 2}(\QQ)$. This coincides with the embedding denoted $\xi$ in \cite[\S 4.3]{BDP}.
  \end{rmk}

  Clearly, the algebra embedding corresponding to $g \cdot \tau$ is $g \iota g^{-1}$. We can also regard $\iota$ as an embedding of algebraic groups $H \into \GL_2$, where $H$ is the torus $\Res_{\cK/\QQ}(\GG_m)$.
%
%  \begin{rmk}
%   Note that if $\cK = \QQ(\sqrt{-D})$, the point $\tau = \sqrt{-D}$ corresponds to the embedding $\iota: a + b \sqrt{-D} \mapsto \mtwo a {-Db} b a$.
%  \end{rmk}

  \begin{notation}
   We let $\tau^* = -1/\bar\tau$.
  \end{notation}

  Clearly, the vector $v = \begin{pmatrix}\tau \\ 1 \end{pmatrix}$ is an eigenvector for $\iota(\cK^\times)$ acting via the standard representation of $\GL_2(\cK)$ on $\cK^2$: we have $\iota(u)\cdot  v = u v$ for all $u \in \cK^\times$. On the other hand, if $(\cK^2)^\vee$ denotes the dual of the standard representation of $\GL_2(\cK)$, then $v^* = \begin{pmatrix}\tau^* \\ 1 \end{pmatrix}$ is an eigenvector for $\iota(\cK^\times)$ acting on $(\cK^2)^\vee$, i.e. ${}^t\iota(u)^{-1} \cdot v^* = u^{-1}  v^*$.

 \subsection{CM points}

  Recall that the modular curve $Y_1(N)$ can be defined as the moduli space of pairs $(E, P)$, where $E$ is an elliptic curve and $P$ a section of exact order $N$ (both defined over some $\ZZ[1/N]$-algebra). The identification
  \( \Gamma_1(N) \backslash \mathbf{H} \cong Y_1(N)(\CC), \)
  where $\mathbf{H}$ is the upper half-plane, is given by mapping $\tau \in \mathbf{H}$ to the pair $\left(\CC/ \ZZ\tau + \ZZ, \tfrac{1}{N} \bmod \ZZ\tau + \ZZ\right)$.

  Let $A$ denote the elliptic curve $\CC/ \OK$, and let $t$ denote a choice of generator of $\fN^{-1} / \OK$. Then the pair $(A, t)$ defines a point of $Y_1(N)(\CC)$. This point is defined over the ray class field $F$ of $\cK$ modulo $\fN$ (see e.g.~\cite[\S 15.3.1]{kato04}).

  \begin{lemma}
   \label{lem:iota}
   Let $\mathbf{H}$ be the upper half-plane. There exists a point $\tau \in \mathbf{H} \cap \cK$ such that
   \begin{enumerate}[(i)]
    \item the $\Gamma_1(N)$-orbit of $\tau$ represents the pair $(A, t)$,
    \item $\tau$ generates the quotient of additive groups $(\OK \otimes \Zp)/ \Zp$,
    \item $\tau$ is a unit at all primes above $p$.
   \end{enumerate}
  \end{lemma}

  \begin{proof}
   If $\tau$ is any generator of the quotient group $\OK/ \ZZ$, then $\tau$ certainly generates $(\OK\otimes \Zp) / \Zp$. If $p$ is inert in $\cK$ then it is immediate that $\tau$ is a unit above $p$; if not, we may achieve this by replacing $\tau$ with $\tau + n$ for some $n \in \ZZ$ (using the fact that by assumption $p \ne 2$).

   This gives a $\tau$ which satisfies (ii) and (iii), and which represents the $\operatorname{SL}_2(\ZZ)$-orbit of $(A, t)$. Moreover, this is also true if we replace $\tau$ with $\gamma \cdot \tau$, for any $\gamma \in \Gamma(p)$, where $\Gamma(p)$ is the principal congruence subgroup of level $p$ in $\operatorname{SL}_2(\ZZ)$. Since $p \nmid N$, $\Gamma(p)$ acts transitively on $\operatorname{SL}_2(\ZZ) / \Gamma_1(N)$, so we can choose a $\tau$ which represents $(A, t)$.
  \end{proof}

  \begin{notation}
   \label{notation:CMpts}
   We fix (for the remainder of this paper) a $\tau$ satisfying the conditions of the lemma, and let $\iota: H \into \GL_2$ be the corresponding embedding. For $m \ge 0$, we let $\tau_m = p^{-m} \tau$, and we let $\iota_m = \mtwo {p^{-m}}001 \iota \mtwo {p^m}001$ be the embedding corresponding to $\tau_m$. We let $A_m = \CC / (\ZZ \tau_m + \ZZ)$, $\phi_m: A \to A_m$ the canonical cyclic $p^m$-isogeny, and $t_m$ the $N$-torsion point $\phi_m(t_A)$ of $A_m$, so that $\Gamma_1(N) \cdot \tau_m$ represents the pair $(A_m, t_m)$.
  \end{notation}

  Note that the endomorphism ring of $A_m$ is $\cO_{p^m}$, so $(A_m, \phi_m)$ defines an element of the set $\mathrm{Isog}^\fN_{p^m}(A)$ in the notation of \cite[\S 1.4]{BDP}. The pair $(A_m, \phi_m)$ is defined over the field $F_m$ of Definition \ref{def:fields}.

  Observe that the Tate module $T_p(A)$ is canonically identified with $\OK\otimes\Zp$. From the main theorem of complex multiplication, we have the following characterisation of the \'etale cohomology of $A$:

  \begin{proposition}
   The action of $\Gal(\Qb/F)$ on $T_p(A) \otimes \Qb_p$ is given by $\sigma^{-1} \oplus \bar\sigma^{-1}$, where $\sigma$ and $\bar\sigma$ are regarded as Galois characters as in \S\ref{sect:galreps}.
  \end{proposition}

  Since $\sigma$ and $\bar\sigma$ are unramified outside $p$ (and $A$ is independent of $p$), the Neron--Ogg--Shafarevich criterion implies that $A$ has good reduction at every prime of $F$. In particular, $(A, t)$ defines a point on the integral model of $Y_1(N)$ over $\cO_F[1/N]$.

 \subsection{Algebraic representations}

  If $E$ is any commutative ring, we write $\Sym^k E^2$ for the left representation of $\GL_2(E)$ afforded by the space of homogenous polynomials of degree $k$ over $E$ in two variables $X, Y$, with $\mtwo a b c d \cdot f = f\left((X, Y) \cdot \mtwo a b c d\right) = f(aX + cY, bX + dY)$. The dual of $\Sym^k E^2$ as a $\GL_2(E)$-representation is given by $\TSym^k( (E^2)^\vee)$, where $(E^2)^\vee$ is the dual of the standard representation, and $\TSym^k$ denotes symmetric tensors (see (see \cite[\S 2.2]{kings15} or \cite[\S 2.2]{KLZ2} for further details).

 We now assume that $E$ is a field, and that there exists an embedding $\sigma: \cK\into E$ (and we fix such a choice). Then $\sigma$ and its conjugate $\bar\sigma$ are 1-dimensional representations of $H$ over $E$.

 \begin{definition}
  We let $G = \GL_2 \times H$, and for integers $a, b \ge 0$, we define $V_{a,b}$ to be the following representation of $G$ over $E$:
  \[ V_{a,b}=\Sym^{a+b}(E^2) \boxtimes \left(\sigma^{-a} \otimes \bar{\sigma}^{-b}\right). \]
 \end{definition}

 We write $\delta$ for the diagonal embedding $(\iota, \mathrm{id}): H \into G$, where $\iota$ is the embedding $H \into\GL_2$ fixed in Notation \ref{notation:CMpts}, and similarly $\delta_m = (\iota_m, \mathrm{id})$. We write $\delta_m^*$ for the restriction of representations from $G$ to $H$ via $\delta_m$. The following computation is straightforward:

 \begin{proposition}
  For any $m$, the representation $\delta_m^*(V_{a,b}^\vee)$ has a unique summand isomorphic to the trivial representation of $H$.\qed
 \end{proposition}

 If $\iota$ is given by $\tau \in \PP^1_{\cK}$ as above, then we let $e_m = \begin{pmatrix}\sigma(\tau_m^*) \\ 1 \end{pmatrix} \in (E^2)^\vee$, where $\tau_m^* = p^m \tau^*$. On this vector $e_m$, the group $\iota_m(H)$ acts as $\sigma^{-1}$. Hence, for every $a, b \ge 0$, we can consider the vector
 \[ e^{[a, b]}_m = (e_m)^{\otimes a} \cdot \left( \overline{e_m}\right)^{\otimes b} \in \TSym^{a+b} ((E^{2})^\vee) = \left(\Sym^k E^2\right)^\vee, \]
 where $\cdot$ denotes the symmetrised tensor product in the algebra $\TSym^\bullet( (E^{2})^\vee)$. The vector $e^{[a, b]}_m$ transforms under $\delta_m(H)$ via $\sigma^{-a} \bar\sigma^{-b}$, so it is a basis vector of $\delta_m^*(V_{a,b}^\vee)$.

 \subsection{Heegner classes}
 \label{sect:pushfwd}
  As in \cite[\S 2.3]{BDP}, the pair $(\phi_m, A_m)$ determines a cycle
  \[ \Delta_{\phi_m} \in \epsilon_W \epsilon_A \CH^{k+1}\left( (W_k \times A^k)_{F_m}\right)_{\QQ},\]
  for each $k \in \ZZ_{\ge 0}$, where $W_k$ is a compactification of the fibre product $\cE^k$ of the universal elliptic curve over $Y_1(N)$, $\epsilon_W$ and $\epsilon_A$ are certain idempotents defined in \emph{op.cit.}, and $F_m$ is as in Notation \ref{notation:CMpts}. Via pullback, we regard $\Delta_{\phi_m}$ simply as a cycle on $\cE^k \times A^k$. In the language of relative Chow motives, we have
  \[  \epsilon_W \epsilon_A \CH^{k+1}\left( (\cE^k \times A^k)_{F_m}\right) =
  H^2_{\mathrm{mot}}\left( Y_1(N)_{F_m}, \TSym^k(h^1(\cE)(1)) \otimes \TSym^k(h^1(A))(1)\right),\]
  where $h^1(\cE) \in \operatorname{CHM}(Y_1(N)_{\QQ})$ is the degree 1 part of the relative motive of $\cE / Y_1(N)$, and similarly for $A$.

  If we choose $a,b \ge 0$ with $a + b = k$, then inside $\TSym^k(h^1(A)) \otimes \cK$ is a rank-1 direct summand $h^{(a, b)}(A)$ on which the complex multiplication action of $\OK$ is given by $[x]^* = x^a \bar{x}^b$. We write
  \[ z^{[a,b]}_{\operatorname{mot}, m} \in H^2_{\mathrm{mot}}\left( Y_1(N)_{F_m}, \TSym^k(h^1(\cE)(1)) \otimes h^{(a, b)}(A)(1)\right)\]
  for the projection of $\Delta_{\phi_m}$ to this direct summand.

  We can reinterpret this more slickly as follows. Let us write $S_m$ for the canonical model over $\cK$ of the Shimura variety for $H$ of level $U_{\fN, p^m} = \{ x \in \widehat{\cO}_{p^m}^\times: x = 1 \bmod \fN\}$. Since our Shimura data for both $H$ and $\GL_2$ are of PEL type, there is a functor from algebraic representations of $G$ to relative Chow motives on $Y_1(N) \times S_m$, compatible with tensor products and duals, constructed by Ancona \cite{ancona,LSZ:gsp4}; and it follows from the definition of this functor that $\TSym^k(h^1(\cE)(1)) \otimes h^{(a,b)}(A)$ is the relative motive associated to $V_{a,b}^\vee$. A result due to Torzewski \cite{torzewski19} shows that this functor is natural with respect to morphisms of PEL data, so we have a diagram
  \[
   \begin{tikzcd}
   \operatorname{Rep}_{\cK}(G) \rar["\delta_m^*"] \dar["\operatorname{Anc}_G" left]& \operatorname{Rep}_{\cK}(H)\dar["\operatorname{Anc}_H"]\\
   \operatorname{CHM}_{\cK}(Y_1(N) \times S_m) \rar["\delta_m^*"]
   & \operatorname{CHM}_{\cK}(S_m)
   \end{tikzcd}
  \]
  where top arrow $\delta_m^*$ is restriction of algebraic representations, and the bottom arrow is pullback of relative Chow motives. For any $a, b \ge 0$, let $\cV_{a,b}$ be the image of $V_{a, b}$ under Ancona's functor. Since $S_m$ has codimension 1 in $Y_1(N) \times S_m$, the pushforward (Gysin) map $\delta_{m*}$ gives a morphism
  \begin{align*}
   H^0_\mathrm{mot}(S_m, \delta_m^*(\cV_{a,b}^\vee )) \xrightarrow{\delta_{m*}}& H^2_\mathrm{mot}(Y_1(N)_{\cK} \times S_m, \cV_{a,b}^\vee (1))\\
   =& H^2_{\mathrm{mot}}\left(Y_1(N)_{F_m}, \TSym^k(h^1(\cE)(1)) \otimes h^{(a, b)}(A)(1) \right),
  \end{align*}
  for any $a, b$, where the last isomorphism comes from the fact that $S_m$ is isomorphic as a $\cK$-variety to the $\Gal(F_m / \cK)$-orbit of $\tau_m$. The left-hand side is just $(\cV_{a,b}^\vee)^{\delta_m(H)}$, and the pushforward of our basis vector $e_m^{[a, b]}$ is exactly the Heegner class.

  If $L$ denotes a $p$-adic field with an embedding $\sigma: \cK \into L$, then the above motivic cohomology groups map naturally to their \'etale analogues with coefficients in $L$; and the $p$-adic realisation of $\cV_{a, b}$ is exactly the lisse \'etale $L$-sheaf on $\Sh_G(U)$ associated to the representation $V_{a, b}$ of $G(\Qp)$. (We shall abuse notation a little by using the same symbol $\cV_{a, b}$ both for the relative Chow motive and its \'etale realisation.) We write
  \[
   z_{\et, m}^{[a,b]} \in H^2_{\et}\left(Y_1(N)_{F_m}, \cV_{a,b}^\vee(1) \right)
  \]
  for the \'etale realisation of $z^{[a,b]}_{\operatorname{mot}, m}$. The realisation map is compatible with pushforward maps, so we deduce that $z_{\et, m}^{[a,b]}$ is simply the pushforward via $\delta_m$ of the class in $H^0_{\et}(S_m, \delta_m^* \cV_{a,b}^\vee)$ given by $e_m^{[a,b]}$.

  \begin{rmk}
   There is a corresponding description of the realisations of the Heegner class in other cohomology theories admitting functorial realisation maps from motivic cohomology, such as absolute Hodge cohomology, or $p$-adic syntomic cohomology.
  \end{rmk}

 \subsection{Projection to eigenspaces}
  \label{sect:projn}

  Since $Y_1(N)$ is affine, its base-change to $\Qb$ has cohomological dimension 1, so the Hochschild--Serre spectral sequence gives a natural map
  \begin{align*}
  H^2_{\et}\left(Y_1(N)_{F_m}, \cV_{a,b}^\vee(1) \right)& \to H^1\left(F_m, H^1_{\et}(Y_1(N)_{\Qb}, \cV^\vee_k(1)) \otimes \sigma^a \bar{\sigma}^b\right) \\
  &= H^1\left(\cK, H^1_{\et}(Y_1(N)_{\Qb}, \cV^\vee_k(1)) \otimes \Ind_{F_m}^\cK  \sigma^a \bar{\sigma}^b\right)
  \end{align*}
  where $\cV_k$ is the \'etale $\Qp$-sheaf associated to the representation $\Sym^k( \Qp^2 )$ of $\GL_2(\Qp)$, and $\sigma, \bar{\sigma}$ are interpreted as characters $\Gal(\cK^{\ab} / F) \to \Qb{}_p^\times$ via the Artin map.

  As the groups $H^1\left(\cK, -\right)$ are not in general finite-dimensional, it is convenient to introduce the following abuse of notation. Let $\Sigma$ be the set of primes of $\cK$ dividing $pN$. Then all the modular curves and Kuga--Sato varieties that we consider have smooth models over the ring $\OK[1/\Sigma]$, so the elements we consider will in fact land in the finite-dimensional subspace $H^1(\OK[1/\Sigma], -)$ (the cohomology unramified outside $\Sigma$). \emph{We shall abusively write $H^1(\cK, -)$ when we mean $H^1(\OK[1/\Sigma], -)$ henceforth}, and similarly for finite extensions of $\cK$.

  \begin{definition}
   If $f$ is an eigenform of level $N$, weight $k + 2$ and character $\veps$, then for $m \ge 0$ and $0 \le j \le k$ we define
   \[
    z_{\et, m}^{[f, j]} \coloneqq \mathrm{pr}_f\left(z_{\et, m}^{[k-j, j]}\right) \in H^1\left(F_m, V_p(f)^*(\sigma^{k-j} \bar{\sigma}^j) \right).
   \]
  \end{definition}

  Now suppose $(f, \chi)$ is a Heegner pair of finite type $(p^m, \fN, \veps)$ and $\infty$-type $(a,b)$. Then $\chi$ gives an extension of $\sigma^a \bar\sigma^b$ to $\Gal(\Qb /\cK)$.

  \begin{proposition}
   \label{prop:invariant}
   The class $z_{\et, m}^{[f, j]}$ lies in $H^1\left(F_m, V_p(f)^*(\chi) \right)^{\Gal(F_m / \cK_m)} \cong H^1\left(\cK_m, V_p(f)^*(\chi) \right)$.
  \end{proposition}

  \begin{proof}
   The 0-dimensional variety $S_m$ has an action of $\widehat{\cO}_{p^m}^\times / U_{\fN, p^m} \cong (\ZZ / N \ZZ)^\times$; and the embedding $\delta_m$ intertwines this with the action on $Y_1(N) \times S_m$ given by $(\langle a \bmod \fN \rangle, a)$, where $\langle a \bmod \fN\rangle$ denotes the action of $(\ZZ/N\ZZ)^\times$  on $Y_1(N)$ via the diamond operators. Since $\chi$ restricts to $\veps \sigma^a \bar\sigma^b$ on $\Gal(\cK^{\ab} / \cK_m)$, this implies that $z_{\et, m}^{[f, j]}$ is invariant under this group when we twist the Galois action by $\veps$. The second isomorphism is an easy consequence of the inflation-restriction exact sequence (since we are working with $\Qp$ coefficients).
  \end{proof}

  \begin{proposition}
   For $(f, \chi)$ a Heegner pair of finite type $(p^m, \fN, \veps)$ and $\infty$-type $(k-j,j)$, we define
   \[ z_{\et}^{[f, \chi]} \coloneqq \operatorname{norm}^{\cK_m}_{\cK}\left(z_{\et, m}^{[f, j]}\right) \in H^1\left(\cK, V_p(f)^*(\chi) \right).
   \]
  \end{proposition}

  For applications to $p$-adic deformation, it is convenient to extend this by defining cohomology classes on the modular curves $Y_1(N(p^n))$ of level $\Gamma_1(N) \cap \Gamma_0(p^n)$, for any $n \ge 0$. If $m \ge n$, then the point on $Y_1(N(p^n))$ corresponding to $\tau_m$ is defined over $F_m$; using this modular curve in place of $Y_1(N)$ in the above constructions, we obtain a class
  \begin{equation}
  \label{eq:Zmn}
   Z_{\et, m, n}^{[a,b]} \in H^2_{\et}\left(Y_1(N(p^n))_{F_m}, \cV_{a,b}^\vee(1) \right);
  \end{equation}
  note that $Z_{\et, m, 0}^{[a,b]}$ is the class $z_{\et, m}^{[a,b]}$ defined above. This definition does not make sense for $m < n$, but one can show that if $n \ge 1$ the points $Z_{\et, m, n}^{[a,b]}$ satisfy
  \[
   \operatorname{norm}^{F_m}_{F_{m+1}}\left(Z_{\et, m+1, n}^{[a,b]}\right) = U_p' \cdot Z_{\et, m, n}^{[a,b]},
  \]
  where $U_p'$ is a Hecke operator (the transpose of the usual $U_p$ with respect to Poincar\'e duality); see Proposition \ref{prop:normrelation} below. So we can \emph{define} $Z^{[a,b]}_{\et, m, n}$ for $m < n$ by this relation, after projecting to the maximal direct summand of $H^2_{\et}\left(Y_1(N(p^n))_{F_m}, \cV_{a,b}^\vee(1) \right)$ on which $U_p'$ acts invertibly (the sum of the generalised eigenspaces with nonzero eigenvalue). In particular, we can thus define classes $z_{\et, m}^{[f_\alpha, j]}$, for $f_\alpha$ a $p$-stabilised eigenform of level $\Gamma_1(N) \cap \Gamma_0(p)$ and $m \ge 0$, and $z_{\et}^{[f_\alpha, \chi]}$ for $(f, \alpha, \chi)$ a $p$-stabilised Heegner pair, in the same way as above.

  \subsubsection{Relating \tp{p}-stabilised and non-\tp{p}-stabilised classes}If $f_\alpha$ is a $p$-stabilisation of a prime-to-$p$ level eigenform $f$, then $V_p(f)^*$ and $V_p(f_\alpha)^*$ are isomorphic as abstract Galois representations, but realised differently as quotients of cohomology of the tower of modular curves. An explicit isomorphism between the two is given by the map $(\Pr_\alpha)_*: V_p(f_\alpha)^* \to V_p(f)^*$ defined in \cite[Definition 5.7.5]{KLZ2}, and a straightforward check gives the following relation, analogous to Theorem 5.7.6 of \emph{op.cit.}:

  \begin{proposition}
   \label{prop:eulerfactor}
   We have
   \[ (\sideset{}{_\alpha}\Pr)_*(z_{\et}^{[f_\alpha, \chi]}) = \cE_p(f_\alpha, \chi) \cdot z_{\et}^{[f, \chi]},\]
   where $\cE_p(f_\alpha, \chi)$ is a non-zero scalar. If $\chi$ is ramified at $p$ this scalar is 1; if $\chi$ is unramified, $\cE_p(f_\alpha, \chi)$ is given by $(1 - \frac{\chi(\fp)}{\alpha}) (1 - \frac{\chi(\fpb)}{\alpha})$ if $p = \fp\fpb$ is split, $(1 - \frac{\chi(p)}{\alpha^2})$ if $p$ is inert, and $(1 - \frac{\chi(\fp)}{\alpha})$ if $p = \fp^2$ is ramified.\qed
  \end{proposition}
%
%  \begin{proof}
%   We recall that $(\sideset{}{_\alpha}\Pr)_* = \operatorname{pr}_{1, *} - \tfrac{\beta}{p^{k+1}}\operatorname{pr}_{2, *}$, where $\operatorname{pr}_1$ is the natural degeneracy map and $\operatorname{pr}_{2}$ corresponds to multiplication by $p$ on the upper half-plane. For $m \ge 1$, we have $\operatorname{pr}_{1, *}\left( Z_{\et, m}^{[a,b]}\right)=  z_{\et, m}^{[a,b]}$, and one sees easily that $\operatorname{pr}_{2, *}\left(Z_{\et, m}^{[a,b]}\right) = p^k z_{\et, {m-1}}^{[a,b]}$ and hence it maps to zero under projection to the $\chi$-eigenspace for $\chi$ primitive of conductor $m$.
%
%   For $m = 0$ the computation is slightly more intricate. One checks that
%   \[
%    \left(\operatorname{norm}_{F_0}^{F_1}\circ \operatorname{pr}_{2, *}\right) Z_{\et, 1}^{[a,b]} = (p + 1) p^k \cdot z_{\et, 0}^{[a,b]}.
%   \]
%   If $p$ is inert, then one computes that $\left(\operatorname{norm}_{F_0}^{F_1}\circ \operatorname{pr}_{1, *}\right) Z_{\et, 1}^{[a,b]} = T_p' \cdot z_{\et, 0}^{[a,b]}$, so that
%   \[ \left(\operatorname{norm}_{F_0}^{F_1}\circ \sideset{}{_{\alpha*}}\Pr\right) Z_{\et, 1}^{[a,b]} = \left(T_p' - \tfrac{p+1}{p} \beta\right) z_{\et, 0}^{[a,b]}, \]
%   whose projection to the $(f, \chi)$-eigenspace is $(\alpha - \frac{\beta}{p}) z_{\et}^{[f, \chi]} = \alpha \cE_p(f_\alpha, \chi)z_{\et}^{[f, \chi]}$ as required. The case of $p$ split or ramified is similar.
%  \end{proof}

\section{Interpolating coefficient systems}

 \subsection{Spaces of distributions}

  Let $L$ be a finite extension of $\Qp$, and $k \ge 0$ an integer. Let $T = \cO_L^2$, considered as a left $\GL_2(\cO_L)$-module via the defining representation of $\GL_2$; and recall the explicit model of the $\GL_2(\cO_L)$-module $\Sym^k T$ described above. As we have seen, the dual of $\Sym^k T$ is the module $\TSym^k( T^\vee)$ of symmetric tensors of degree $k$ over $T^\vee$.

  \begin{definition} Let $n \ge 1$.
   \begin{enumerate}
    \item Define a compact open subgroup $K_0(p^n) \subset \GL_2(\Zp)$ by
    \[ K_0(p^n) \coloneqq \left\{ \mtwo a b c d \in \GL_2(\Zp) : c = 0 \bmod p^n\right\}. \]
    \item Let $A_{k, n}$ denote the space of power series in a variable $Z$ (with coefficients in $L$) which converge on $p^n \cO_{\CC_p}$, equipped with a left action of $K_0(p^n)$ via
    \begin{equation}
     \label{eq:wtkaction}
     \mtwo a b c d \cdot f =
     \left(bZ + d\right)^k f\left( \frac{aZ + c}{bZ + d} \right).
    \end{equation}
    We consider $A_{k, n}$ as a Banach space in the supremum norm (as functions on $p^n \cO_{\CC_p}$), with unit ball the $\cO_L$-lattice $A_{k, n}^\circ$ of functions whose supremum norm is $\le 1$.
    \item We consider $\Sym^k T$ as a submodule of $A_{k, n}^\circ$, via $f(X, Y) \mapsto f(Z, 1)$.
    \item We let $D^\circ_{k, n} = \Hom_{\cO_L}\left(A_{k, n}^\circ, \cO_L\right)$, and $D_{k, n} = D^\circ_{k, n} \otimes L = \Hom_{\mathrm{cts}}(A_{k, n}, L)$. We equip these with the dual action of $K_0(p^n)$ defined by
    \[ (\gamma \cdot \phi)(f) = \phi(\gamma^{-1} f) \]
    for all $\gamma \in K_0(p^n)$, $\phi \in D_{k, n}$, and $f \in A_{k, n}$. Dualising the inclusions of (3) gives us $K_0(p^n)$-equivariant \emph{moment} maps
    \[ \mom^k: D^\circ_{k, n} \to \TSym^k(T^\vee). \]
   \end{enumerate}
  \end{definition}

  Note that the action \eqref{eq:wtkaction} is well-defined, since $c \in p^n \Zp$ and $d \in \Zp^\times$, and hence $Z \mapsto \frac{c + aZ}{d + bZ}$ preserves $p^n \cO_{\CC_p}$.

  We can extend the action of $K_0(p^n)$ on $A_{k,0}$ to the monoid $\Sigma_0(p^n) \subset \GL_2(\Qp)$ generated by $K_0(p^n)$ together with the element $\mtwo p001$, which acts on $A_{k, n}$ as $f(Z) \mapsto f(pZ)$. Hence we obtain a left action on $D_{k, n}$ (or $D^\circ_{k, n}$) of the monoid $\Sigma_0'(p^n)$ generated by $K_0(p^n)$ and $\mtwo {p^{-1}}001$. This is compatible with the specialisations $\mom^{k}$.

 \subsection{Interpolation over weight space}
  \label{sect:weightspace}
  Let $\cW$ denote the rigid space parameterising characters of $\Zp^\times$, as above. We shall interpolate the distribution spaces $D_{k, n}$ over $\cW$, following \cite{andreatta-iovita-stevens:overconvES, hansen, loeffler-zerbes:coleman}.

  More precisely, for $n \ge 1$, let $\cW_n$ be the locus in $\cW$ consisting of ``$n$-analytic'' characters, which are the characters $\kappa$ such that $v_p(\kappa(1 + p) - 1) > \frac{1}{p^{n-1}(p-1)}$, where $v_p$ is the valuation such that $v_p(p) = 1$. (Geometrically, $\cW$ is a union of $p-1$ discs of radius 1, and $\cW_n$ is the union of slightly smaller open discs inside each component of $\cW$.) If $\kappa$ is $n$-analytic, then its restriction to each coset $a + p^n \Zp$ is given by a single convergent power series.

  We let $U$ be an open disc contained in $\cW_n$ and defined over $L$. Associated to $U$ is an algebra $\Lambda_U \cong \cO_L[[u]]$ of functions on $U$ bounded by 1, which is a local ring with maximal ideal $m_U$ and finite residue field; and a universal character $\kappa_U : \Zp^\times \to \Lambda_U^\times$. We define $A_{U, n}^\circ$ to be the space of power series $\sum_{s \ge 0} a_s (\tfrac{Z}{p^n})^s$ with $a_s \in \Lambda_U$ tending to 0 in the $m_U$-adic topology. This space has a right action of $\Sigma_0(p^n)$, via the same formulae as before. (This relies crucially on the fact that $\kappa_U$ is $n$-analytic, which follows from the assumption that $U \subseteq \cW_n$.)

  If we let $D^\circ_{U, n}$ be the set of $\Lambda_U$-linear maps $A_{U, n}^\circ \to \Lambda_U$, then $D^\circ_{U, n}$ is a profinite topological $\Lambda_U$-module with a left action of $\Sigma_0'(p^n)$, equipped with a filtration by open $\Sigma_0'(p^n)$-invariant $\Lambda_U$-submodules, and with $\Sigma_0'(p^n)$-equivariant moment maps
  \[ \mom^k: D^\circ_{U, n} \mathop{\hat{\otimes}}_{(\Lambda_U, k)} \cO_L \to \TSym^k T^\vee \]
  for all $k \in U$.  Cf.~\cite[Lemma 4.2.8]{loeffler-zerbes:coleman}.

 \subsection{The eigen-distribution associated to \tp{H}}

  Recall that we have fixed in Notation \ref{notation:CMpts} a point $\tau \in \mathbf{H} \cap \cK$ which is a unit at the primes above $p$, and defined $\tau^* = -1/\bar{\tau}$. Enlarging $L$ if necessary, we suppose that $L$ contains the image of $\cK$ under our embedding $\Qb \into \Qb_p$, so we have a distinguished embedding $\sigma: \cK \into L$.

  Let $m \ge n$. Since $\sigma(\tau^*) \in \cO_L$, ``evaluation at $\sigma(\tau_m^*) = p^m \sigma(\tau^*)$'' is well-defined as a map $A_{k, n}^\circ \to \cO_L$, i.e.~as an element $\mathbf{e}_{k,0,m} \in D_{k, n}^\circ$.

  \begin{proposition}
   We have $\iota_m\left( (\cO_{p^m} \otimes \Zp)^\times\right) \subseteq K_0(p^n)$, and $\iota_m((\cO_{p^m} \otimes \Zp)^\times)$ acts on $\mathbf{e}_{k,0,m}$ via the character $\sigma^{-k}$. Moreover, the image of $\mathbf{e}_{k,0,m}$ under $\mom^k$ is $e_m^{[k, 0]} \in \TSym^k T^\vee$.
  \end{proposition}

  \begin{proof}
   Unravelling the definitions, if $u \in \cK \otimes \Qp$ maps to $\mtwo a b c d \in \GL_2(\Qp)$ under $\iota_m$, then $a \sigma(\tau_m^*) + c = \sigma(u) \tau_m^*$ and $b \sigma(\tau_m^*) + d = \sigma(u)$. Thus $\tfrac{a\tau_m^* + c}{b \tau_m^* + d} = \tau_m^*$, so for $f \in A_{k, n}^\circ$ we have
   \[ (\mtwo a b c d \cdot f)(\sigma(\tau_m^*)) = \sigma(b\tau_m + d)^k f(\tau_m^*) = \sigma(u)^k f(\sigma(\tau_m^*)).\]
   Hence $\mtwo a b c d \cdot \mathbf{e}_{k,0,m} = \sigma(u)^{-k}\mathbf{e}_{k,0,m}$ as required.

   The restriction of this linear functional $\mathbf{e}_{k, 0, m}$ to $P_k = \Sym^k T$ is given by evaluation of polynomials at $(X, Y) = (\sigma(\tau^*_m), 1)$, which is exactly the definition of $e_m^{[k, 0]}$.
  \end{proof}

  For an open disc $U \subseteq \cW_n$ as above, we have a class $\mathbf{e}_{U, 0, m}$ defined similarly, for any $m \ge n$. This is an $\iota_m\left( (\cO_{p^m} \otimes \Zp)^\times \right)$-invariant vector in $D_{U, n}^\circ$ which transforms under $(\cO_{p^m} \otimes \Zp)^\times$ by the character $\kappa_U \circ \sigma$ (which is well-defined as a character of $\cO_n^\times$, since $\kappa_U$ is $n$-analytic).

 \subsection{Twisting}

  We have succeeded in $p$-adically interpolating the $e_m^{[k, 0]}$, for varying $k \ge 0$. We shall now extend this by interpolating the $e_m^{[k-j, j]}$, for a fixed $j \ge 0$ and varying integers $k \ge j$, via the following simple trick: we consider the tensor product of the vector $\mathbf{e}_{k-j, 0, m} \in (D_{k-j, n}^\circ)$ with $e_m^{[0, j]} \in \TSym^j(T^\vee)$. This gives, clearly, a vector in the space $D_{k-j, n}^\circ \otimes \TSym^j T^\vee$, on which $\iota_m((\cO_{p^m} \otimes \Zp)^\times)$ acts via $\sigma^{-(k-j)} \bar\sigma^{-j}$; and its image under the composite
  \[ D_{k-j, n}^\circ \otimes \TSym^j T^\vee \xrightarrow{\mom^{k-j} \otimes \mathrm{id}}\TSym^k T^\vee \otimes \TSym^j T^\vee \to \TSym^k T^\vee,\]
  where the last map is the symmetrised tensor product, is $e_m^{[k-j, j]}$. This evidently interpolates over discs $U$, as above. More subtly, for any $j \ge 0$ there is a map of $\Sigma_0'(p^n)$-modules
  \[ \Pi_j: D_{U-j, n} \otimes \TSym^j V^\vee \to D_{U, n}, \]
  the ``overconvergent projector''. (This is the map denoted by $\delta_j^*$ in \cite[\S 5.2]{loeffler-zerbes:coleman}, but we use a different notation here to avoid conflict with the morphism denoted $\delta_m$ elsewhere in this paper.) The map $\Pi_j$ does not preserve the natural $\cO_L$-lattices in general, but its denominator is bounded by an explicit function of $n$ and $j$.

  \begin{definition}
   We define
   \[
    \mathbf{e}_{U, j, m} \coloneqq \Pi_j\left(\mathbf{e}_{U-j, 0, m} \otimes e_m^{[0, j]}\right) \in D_{U, n}.
   \]
  \end{definition}

  By construction, this vector transforms under $(\cO_{p^m} \otimes \Zp)^\times$ via the character $\sigma^{-(\kappa_U - j)}\bar\sigma^{-j}$, and its image under $\mom^{k}$ for $k \in U$ is given by
  \[ \mom^{k}\left( \mathbf{e}_{U, j, m} \right) =
   \begin{cases}
    e_m^{[k-j, j]} & \text{if $k \ge j$,} \\
    0 & \text{if $k < j$.}
   \end{cases}
  \]

  \begin{rmk}
   The class $\mathbf{e}_{U, j, m}$ is independent of $n$: more precisely, the classes with the same name for different $n$'s are compatible under the natural maps $D_{U, n} \to D_{U, n-1}$.

   One can make the value of $\mathbf{e}_{U, j, m}$ on $f \in A_{U, n}$ explicit using the formulae of \emph{op.cit.}: it is a linear combination of the first $j$ derivatives of $f$ at $p^m \tau^*$. However, we shall not need this explicit formula here.
  \end{rmk}

  The following lemma will be crucial in \S\ref{sect:twistcompat} below:
  \begin{lemma}
   \label{lem:congruence}
   Let $h \ge 0$. There is a constant $C$, depending on $n$ and $h$ but independent of $m$, such that
   \[
    \sum_{j = 0}^h (-1)^j \binom{\nabla-j}{h-j} \mathbf{e}_{U, j, m} \in C p^{hm} D^\circ_{U, n},
   \]
   where $\nabla \in \Lambda_U[1/p]$ acts in weight $k$ as multiplication by $k$.
  \end{lemma}

  \begin{proof}
   We note first that $e_m - \bar{e}_m = p^m \cdot \begin{pmatrix} \tau^* - \bar{\tau}^* \\ 0 \end{pmatrix}$ as elements of $T^\vee \cong (\cO_L)^2$. Thus $e_m - \bar{e}_m \in p^m T^\vee$; so
   \[
     (e_m - \bar{e}_m)^{\otimes h} = \sum_{j = 0}^h (-1)^j e_m^{[h-j, j]} \in p^{mh} \TSym^h T^\vee.
   \]
   On the other hand, we know there is a constant $C$ (depending only on $n$ and $h$) such that
   \[ \Pi_h\left (D^{\circ}_{U-h, n} \otimes \TSym^h T^\vee\right)\subseteq C D^\circ_{U, n}.\]
   Hence we obtain
   \[ \sum_{j = 0}^h (-1)^j \Pi_j\left( \mathbf{e}_{U-h, 0, m} \otimes e_m^{[h-j, j]}\right) \in C p^{mh} D^\circ_{U, n}.\]
   Using the identity
   \[ \Pi_h\left(\mathbf{e}_{U-h, 0, m} \otimes e_m^{[h-j, j]}\right) = \binom{\nabla-j}{h-j} \mathbf{e}_{U, j, m} \]
   (cf.~Lemma 5.1.5 of \cite{loeffler-zerbes:coleman}), we deduce the result.
  \end{proof}

\section{Families of Heegner classes}

 \subsection{Families over weight space}

  As above, let $U$ be an open disc in weight space defined over $L$, and let $\kappa_U$ be the associated universal weight; and suppose that $U \subseteq \cW_n$, for some $n \ge 1$, so that $\kappa_U$ is $n$-analytic in the sense of \S \ref{sect:weightspace}.

  Since $D_{U, n}^\circ$ is a profinite left $K_0(p^n)$-module, we can interpret it as an \'etale sheaf on the modular curve $Y_1(N(p^n))$. For each $m \ge n$ and $j \ge 0$, we can regard $\mathbf{e}_{U, j, m}$ as a section
  \[ \mathbf{e}_{U, j, m} \in H^0_{\et}\left(S_m, \delta_m^*\left(D_{U, n}^\circ \otimes \sigma^{\kappa_U-j} \bar{\sigma}^j\right) \right) \]
  in the notation of \S \ref{sect:pushfwd}, whose image under the moment map $\mom^k$, for any $k \in U$ with $k \ge j$, is the \'etale realisation of $e^{[k-j, j]}_m$. We can thus define
  \[ z_{U, m, n}^{[j]} = (\delta_m)_*\left(\mathbf{e}_{U, j, m}\right) \in H^2_{\et}\left(Y_1(N(p^n))_{F_m}, D_{U, n}(1) \otimes \sigma^{(\kappa_U-j)}\bar\sigma^j\right), \]
  and we automatically obtain the following interpolation property:

  \begin{proposition}
   For every $m \ge n$, and every $k \in U \cap \ZZ_{\ge 0}$, we have
   \[
    \mom^k\left(z_{U, m, n}^{[j]}\right) =
    \begin{cases}
     Z^{[k-j, j]}_{\et, m, n} & \text{if $k \ge j$},\\
     0 & \text{if $k < j$},
    \end{cases}
   \]
   where $Z^{[k-j, j]}_{\et, m, n}$ is as defined in \eqref{eq:Zmn}. \qed
  \end{proposition}

  These classes have the following straightforward but crucial norm-compatibility property. Let $\hat Y$ be the modular curve of level $\{ g \in \Gamma_1(N): g = 0 \bmod \mtwo{*}{p}{p^n}{*}\}$; and consider the diagram of modular curves
  \[
   \begin{tikzcd}
    \hat{Y}\rar["\Phi" above, "\cong" below] \dar["\hat{\operatorname{pr}}" left]&
    Y_1(N(p^{n+1}))\dar["\operatorname{pr}" right]\\
    Y_1(N(p^n))
    & Y_1(N(p^n))
   \end{tikzcd}
  \]
  where $\Phi$ is given by the action of $\mtwo p 0 0 1 \in \GL_2(\Qp)$ (which corresponds to $\tau \mapsto \tau/p$ on the upper half-plane). On the cohomology of the $\Qp$-sheaves $\cV_k$ and $\cV_k^\vee$, this diagram gives rise to two Hecke operators: the operator $(\hat{\operatorname{pr}})_* \circ \Phi^* \circ (\operatorname{pr})^*$, which is the usual Hecke operator $U_p$; and the operator $U_p' = (\operatorname{pr})_* \circ (\Phi^{-1})^* \circ (\hat{\operatorname{pr}})^*$, which is the dual of $U_p$ under Poincar\'e duality. See \cite[\S 2.4]{KLZ2} for further details. The action of $U_p'$ makes sense on the cohomology of the sheaves $D_{k, n}$, compatibly with the moment maps, since pullback by $\Phi^{-1}$ corresponds to acting on the sheaf by $\mtwo {p^{-1}} 0 0 1 \in \Sigma_0'(p)$.

  \begin{proposition}
   \label{prop:normrelation}
   For all $m \ge n \ge 1$, we have
   \[ \operatorname{norm}_{F_m}^{F_{m+1}} \left(z_{U, m+1, n}^{[j]}\right) = U_p' \cdot z_{U, m, n}^{[j]}.\]
  \end{proposition}

  \begin{proof}
   One computes that the CM point at level $\hat Y$ corresponding to $\tau_m$ is defined over $F_{m+1}$, and its orbit under the action of $\Gal(F_{m+1} / F_m)$ is exactly the preimage under $\hat{\operatorname{pr}}$ of the $\Gamma_1(N(p^n))$-orbit of $\tau_m$. Applying the same constructions as before to this CM point on $\hat Y$ gives a class
   \[
    \hat{z}_{U, m, n}^{[j]} \in H^2_{\et}\left(\hat Y_{F_{m+1}}, D_{U, n}(1) \otimes \sigma^{(\kappa_U-j)}\bar\sigma^j\right)
   \]
   satisfying
   \[ \operatorname{norm}_{F_m}^{F_{m+1}}\left(\hat z_{U, m, n}^{[j]}\right) = \hat{\operatorname{pr}}^*\left( z_{U, m, n}^{[j]}\right).\]
   Hence we have
   \[
    U_p' \cdot z_{U, m, n}^{[j]} =\operatorname{norm}_{F_m}^{F_{m+1}}\left( (\operatorname{pr})_* (\Phi^{-1})^* \hat z_{U, m, n}^{[j]}\right).
   \]
   The image under $\Phi$ of the CM point $\tau_m$ is $\tau_{m+1}$; and the action of $\Phi^{-1}$ on the sheaf sends $\mathbf{e}^{[j]}_{U, m, n}$ to $\mathbf{e}^{[j]}_{U, m+1, n}$. So we can conclude that
   \[  (\operatorname{pr})_* (\Phi^{-1})^* \hat z_{U, m, n}^{[j]} = z_{U, m+1, n}^{[j]}.\]
   Combining these last two formulae gives $\operatorname{norm}_{F_m}^{F_{m+1}} \left(z_{U, m+1, n}^{[j]}\right) = U_p' \cdot z_{U, m, n}^{[j]}$, as required.
  \end{proof}

 \subsection{Projection to a Coleman family}

  Let $f_0$ be a level $N$ newform of some weight $k_0 + 2 \ge 2$, and $\alpha_0$ a root of its Hecke polynomial at $p$. We impose the technical condition that the $p$-stabilisation $(f_0)_{\alpha_0}$ be a \emph{noble eigenform} in the sense of \cite{hansen}, which is automatic if $\alpha_0$ has valuation $< k_0 + 1$ and is distinct from\footnote{This condition $\alpha_0 \ne \beta_0$ is known to be automatically satisfied if $k_0 = 0$, and for all $k_0 \ge 0$ assuming the Tate conjecture \cite{colemanedixhoven98}.} the other root $\beta_0$.

   Then there exists an \emph{affinoid} disc $V \ni k_0$ in $\cW$, and a unique Coleman family $\cF$ of eigenforms over $V$, specialising in weight $k_0$ to $(f_0)_{\alpha_0}$ (equivalently, a lifting of $V$ to a neighbourhood of $(f_0, \alpha_0)$ in $\cE(N)$). Shrinking $V$ if necessary, we may arrange that for every $k \in V \cap \ZZ_{\ge 0}$, the specialisation of $\cF$ at $k$ is a $p$-stabilisation $f_\alpha$ of some level $N$ newform $f$ of weight $k+2$.

  As explained in \cite[\S 4.6]{loeffler-zerbes:coleman}, after possibly further shrinking $V$, we may find an \emph{open} disc $U \supset V$ contained in $\cW_1$ with the following property: the Galois representation
  \[ M_{V} \coloneqq H^1_{\et}\left(Y_1(N(p))_{\Qb}, D_{U, 1}(1)\right) \mathop{\hat{\otimes}}_{\Lambda_U[1/p]} \cO(V)\]
  has a direct summand $M_V(\cF)^*$ which is free of rank 2 over $\cO(V)$, and interpolates the (dual) Galois representations of the classical specialisations $f_\alpha$ of $\cF$. See \cite[\S 4.6]{loeffler-zerbes:coleman} for further details.

  Our assumption that all classical-weight specialisations of $\cF$ are classical forms implies that the image of $z_{U, m, 1}^{[j]}$ in $M_V(\cF)^*$ is divisible by $\binom{\nabla}{j}$ (cf.~\cite[Prop 5.2.5]{loeffler-zerbes:coleman}). So we obtain classes
  \[ z_{\cF, m}^{[j]} \in H^1\left(F_m, M_V(\cF)^* \otimes \sigma^{\kappa_V - j} \bar{\sigma}^j\right), \]
  for all $m \ge 1$ and $j \ge 0$, satisfying
  \[ \operatorname{norm}_{F_m}^{F_{m+1}} \left(z_{\cF, m+1}^{[j]}\right) = \alpha_{\cF} \cdot z_{\cF, m}^{[j]}, \]
  where $\alpha_{\cF} \in \cO(V)^\times$ is the $U_p$-eigenvalue of $\cF$. These have the following interpolation property: let $k \in V \cap \ZZ_{\ge j}$, and let $f_\alpha$ be the specialisation of $\cF$ at $k$. Then the image of $z_{\cF, m}^{[j]}$ under specialisation at $k$ is $\frac{1}{\binom{k}{j}} z_{\et, m}^{[f_\alpha, j]}$.

 \subsection{Interpolation in \tp{j}}
  \label{sect:twistcompat}

  Let $F_\infty = \bigcup_{n \ge 1} F_n$, and let $\Gamma_1^\ac = \Gal(F_\infty / F_1) \cong (\cO_p \otimes \Zp)^\times / \Zp^\times$. This is an abelian group isomorphic to the additive group $\Zp$, and $\chi^\ac = \sigma/\bar{\sigma}$ is a character $\Gamma_1^\ac \to \Qb{}^\times_p$ of infinite order. Let $\cW_1^\ac$ denote the rigid-analytic variety parametrising characters of $\Gamma_1^\ac$, which is equipped with a universal character $\kappa^\ac: \Gamma_1^\ac \to \cO(\cW_1^\ac)^\times$. We regard $\ZZ$ as a subgroup of $\cW_1^\ac(\Qb_p)$ via the powers of $\chi^\ac$, so the specialisation of $\kappa^\ac$ at $n \in \ZZ$ is $(\chi^\ac)^n$.
  \begin{proposition}
   There exists a cohomology class
   \[ z_{\cF, \infty} \in H^1\left(F_1, M_V(\cF)^* \otimes \sigma^{\kappa_V} \htimes \cO(\cW_1^\ac)(-\kappa^\ac)\right)\]
   whose image in $H^1\left(F_m, M_V(\cF)^* \otimes \sigma^{\kappa_V-j} \bar{\sigma}^j\right)$, for any $j \ge 0$ and $m \ge 1$, is $\frac{1}{\alpha_{\cF}^m} z_{\cF, m}^{[j]}$.
  \end{proposition}

  \begin{proof}
   Let $\lambda \ge 0$ be such that $p^\lambda$ is the supremum norm of $\alpha_{\cF}^{-1} \in \cO(V)$. We choose an integer $h \ge \lfloor \lambda \rfloor$, and consider the family of cohomology classes defined by
   \[ c_{m, j, h} = \alpha_{\cF}^{-m} \binom{\nabla}{h} z_{\cF, m}^{[j]},\]
   for $0 \le j \le h$ and $m \ge 1$. Here $\nabla \in \cO(V)$ is the element whose value at a character $\kappa: \Zp^\times \to \CC_p^\times$ is given by $\kappa'(1)$ (cf.~\cite[Proposition 5.2.5]{loeffler-zerbes:coleman}); in particular, $\nabla$ takes the value $n$ at the character $x \mapsto x^n$.

   These classes $c_{m, j, h}$ are norm-compatible in $m$. Moreover, if $\|\cdot\|$ denotes the supremum seminorm on $H^1(F_\infty,  M_V(\cF)^*)$ induced by some choice of Banach $\cO(V)$-module norm on $M_V(\cF)^*$, then these classes satisfy the bound
   \[ \left\|\sum_{j = 0}^h (-1)^j \binom{h}{j}\Res_{F_m}^{F_\infty}(c_{m, j, h}) \right\| = O(p^{-(h-\lambda)m}), \]
   by Lemma \ref{lem:congruence}; note that we are using here the identity
   \[ \binom{h}{j} \cdot \binom{\nabla}{h} = \frac{\nabla(\nabla-1) \dots (\nabla-h+1)}{j! (h-j)!} = \binom{\nabla}{j} \binom{\nabla - j}{h-j}, \]
   so $\sum_{j = 0}^h (-1)^j \binom{h}{j}\operatorname{Res}_{F_m}^{F_\infty}(c_{m, j, h})$ is the image of the quantity $\sum_{j = 0}^h (-1)^j \binom{\nabla-j}{h-j}
    \mathbf{e}_{U, j, m}$ considered in the lemma.

   Since $\cF$ is cuspidal, we have $H^0(F_\infty, M_V(\cF)^*) = 0$; so we can now invoke the general construction of \cite[Proposition 2.3.3]{loeffler-zerbes:coleman} to give a class
   \[ c(h) \in H^1(F_1, M_V(\cF)^* \otimes \sigma^{\kappa_V} \hat{\otimes} D_\lambda(\Gamma_1^\ac)(-\kappa^\ac)) \]
   which interpolates these classes, where $D_\lambda(\Gamma_1^\ac)$ is the subspace of $\cO(\cW_1^\ac)$ consisting of distributions of order $\lambda$. Moreover, if we take two different values $h, h'$, then the classes ${\binom{\nabla}{h}} c(h')$ and $\binom{\nabla}{h'} c(h)$ have the same specialisations at locally algebraic characters of degree up to $\min(h, h') \ge \lfloor \lambda \rfloor$, so they must in fact agree.

   We can map $c(h)$ into the slightly larger module $H^1\left(F_1, M_V(\cF)^* \otimes \sigma^{\kappa_V} \htimes \cO(\cW_1^\ac)(-\kappa^\ac)\right)$. This is the sections of a torsion-free sheaf over $V \times \cW_1^\ac$; and the image of $c(h)$ is divisible by $\binom{\nabla}{h}$ (because sufficiently many of its specialisations are so). Hence the quotient $c(h) / \binom{\nabla}{h}$ is well-defined and independent of $h$, and therefore enjoys an interpolating property at characters of all weights $j \ge 0$.
  \end{proof}

 \subsection{Re-parametrisation of weights}

  Since $\sigma^{\kappa_V}$ is not well-defined as a character of $\Gal(\Qb/\cK)$ but only of $\Gal(\Qb / F_1)$, it is convenient to ``re-parametrise'' $V \times \cW_1^\ac$, as follows.

  Consider the group $\Gamma_{\cK, 1} = (\Zp + p \cO_{\cK, p})^\times$, and let $\cW_{\cK, 1}$ denote its character space. Then we have a short exact sequence of abelian profinite groups
  \[ 1 \to \Zp^\times \to \Gamma_{\cK, 1} \to \Gamma_1^\ac \to 1, \]
  (which is, in fact, split, although there is no canonical splitting). There is therefore a natural map $\cW_{\cK, 1} \to \cW$, whose fibres are the orbits of $\cW^\ac_1$. Over $V$, this morphism admits a canonical section (since $1$-analytic characters of $\Zp^\times$ extend to $\Gamma_{\cK, 1}$). Hence the preimage of $V$ in $\cW_{\cK, 1}$ is isomorphic as a rigid-analytic space to $V \times \cW^\ac_1$.

  On the other hand, we have defined above a space $\cW_\cK(\fN, \veps)$, and restriction of characters defines a finite map $\cW_\cK(\fN, \veps) \to \cW_{\cK, 1}$. Some careful book-keeping shows that if $\tilde V$ denotes the preimage of $V \times \cW^\ac_1$ in $\cW_\cK(\fN, \veps)$, and  $\kappa_{\tilde V}$ is the universal character $\Gal(\cK^{\ab}/ \cK) \to \cO(\tilde V)^\times$, then the pullback of $z_{\cF, \infty}$ to $\tilde{V}$ is an element of $H^1\left(F_1, M_{\tilde V}(\cF)^*\right)$,
  where
  \[ M_{\tilde V}(\cF)^* \coloneqq M_V(\cF)^*\otimes_{\cO(V)} \cO(\tilde V)(\kappa_{\tilde V}).\]
  Moreover, the same argument as Proposition \ref{prop:invariant} shows that this class is invariant under $\Gal(F_1 / \cK_1)$, and hence descends to $H^1\left(\cK_1, M_{\tilde V}(\cF)^*\right)$. We define $z_{\cF}$ to be the image of $z_{\cF, \infty}$ under the norm map $H^1\left(\cK_1, M_{\tilde V}(\cF)^*\right) \to H^1\left(\cK, M_{\tilde V}(\cF)^*\right)$. This is the ``universal'' version of the projection to the $\chi$-component in the definition of the \'etale classes, so one has the following interpolation formula:

  \begin{theorem}[Theorem A]
   Let $x$ be a point of $\tilde V$ corresponding to a $p$-stabilised Heegner pair $(f,\alpha, \chi)$, of finite type $(p^m, \fN)$. Then the specialisation of the class $z_{\cF} \in H^1\left(\cK, M_{\tilde V}(\cF)^*\right)$ at $x$ is given by
   \[
    \frac{\cE_p(f_\alpha, \chi)}{\alpha^{m} \binom{k}{j}} \cdot z_{\et}^{[f, \chi]},
   \]
   where $\cE_p(f_\alpha, \chi)$ is as defined in Proposition \ref{prop:eulerfactor}.
  \end{theorem}

  \begin{rmk}
   Recall that $\tilde V$ is the preimage in the 2-dimensional space $\cE_\cK(\fN)$ of a small disc in the $\GL_2/\QQ$ eigencurve $\cE(N)$ around $(f_0, \alpha_0)$. It is not clear to what extent the definition of $z_{\cF}$ can be ``globalized'' over the whole rigid-analytic variety $\cE_\cK(\fN)$.
  \end{rmk}

\section{Local properties of the Heegner classes}
 \label{sect:selmer}

 As a complement to Theorem A, we now study the local properties of the class $z_{\cF}$ at primes of $\cK$. For brevity, throughout this section we shall write $\cM$ for the $\cO(V)$-linear $\Gal(\Qb/\QQ)$-representation $M_V(\cF)^*$, and $\wM$ for the $\cO(\tilde V)$-linear representation $\Gal(\Qb/\cK)$-representation $M_{\tilde V}(\cF)^*$.

 \subsection{Local conditions outside \tp{p}}

  \begin{proposition}
   \label{prop:selmer1}
   For any prime $v \nmid p$ of $\cK$, the image of $z_{\cF}$ under the map
   \[ H^1\left(\cK, \wM\right)  \to
      H^1\left(I_v, \wM\right) \]
   is 0, where $I_v$ is an inertia group at $v$.
  \end{proposition}

  \begin{proof}
   Let $\ell$ be the rational prime below $v$. The claim is automatic if $\ell \notin \Sigma$, where $\Sigma$ was the set of primes fixed in \S \ref{sect:projn}, since we have constructed $z_{\cF}$ using the \'etale cohomology of a smooth $\OK[1/\Sigma]$-scheme.

   So it remains to treat the case of $\ell \in \Sigma$, and since $\ell \ne p$, this implies that $\ell \mid N$. By our Heegner hypothesis, all such primes are split in $\cK$. It follows that $v$ has infinite decomposition group in the anticyclotomic $\Zp$-extension of $\cK$; so $z_{\cF, \infty}$ and hence also $z_{\cF}$, are unramified at $v$ by the same argument as \cite[Proposition 2.4.4]{loeffler-zerbes:coleman}.
  \end{proof}

 \subsection{Phi-Gamma modules and triangulations}

  In order to study the local properties of the Heegner class at $p$, we recall some results on $p$-adic Hodge theory of $\cM$ as a representation of $G_{\Qp} = \Gal(\Qb_p / \Qp)$, following \cite[\S 6]{loeffler-zerbes:coleman}. These results will also be needed in \S\ref{sect:ERL} below.

  We first summarise the concepts of $p$-adic Hodge theory we will use. See e.g.~\cite[\S 2.2]{bellaichechenevier09} for further details.
  \begin{itemize}
   \item For $E$ a finite extension of $\Qp$, let $\cR_E = \mathbf{B}^{\dag}_{\rig, E}$ be the Robba ring. There is a full faithful functor $\Drig$, compatible with tensor products and duals, which embeds the category of $\Qp$-linear representations of $\Gal(\Qb_p/E)$ inside the larger category of $(\varphi, \Gamma)$-modules over $\cR_E$.
   \item The functor $\Drig$ extends naturally to locally free families of Galois representations over rigid spaces $X$ (taking values in locally free sheaves of $(\varphi, \Gamma)$-modules over $\cR_E \htimes \cO(X)$).
   \item There exist cohomology functors $H^i(E, -)$ for $(\varphi, \Gamma)$-modules, compatible with the Galois cohomology functors via $\Drig$. The cohomology groups of a family of $(\varphi, \Gamma)$-modules over $X$ are coherent sheaves on $X$ (by \cite{KPX}), and there is a descent spectral sequence relating these to the cohomology of any specialisation of the family.
   \item The Fontaine functor $\Dcris$, can be defined on $(\varphi, \Gamma)$-modules over $\cR_E$ (consistently with its existing definition for Galois representations), so it makes sense to speak of a $(\varphi, \Gamma)$-module being crystalline.
   \item If $E'$ is an extension of $E$ then $\cR_E$ is a subring of $\cR_{E'}$, and restriction of Galois representations corresponds to base-extension of $(\varphi, \Gamma)$-modules. If $E / \Qp$ is unramified, then $\cR_E = \cR_{\Qp} \otimes E$, with $\varphi$ acting on $E$ as the arithmetic Frobenius (and $\Gamma$ acting trivially on $E$).
  \end{itemize}

  \begin{theorem}[{Ruochuan Liu, \cite[Theorem 0.3.4]{liu15}}]\label{thm:liu}
   If the disc $V \ni k_0$ is sufficiently small, the Galois representation $\cD_p = \Drig(\cM |_{G_{\Qp}})$ over $\cR \htimes \cO(V)$ is ``trianguline'': it admits a submodule $\cD_p^+$ stable under $\varphi$ and $\Gamma$ such that both $\cD_p^+$ and $\cD_p^- \coloneqq \cD_p / \cD_p^+$ are locally free of rank 1.
  \end{theorem}

  Liu's construction also gives a precise description of $\cD_p^+$: it has a basis vector on which $\varphi$ acts as $\veps(p)^{-1} \cdot a_p(\cF)$, and $\Gamma$ acts via the $\cO(V)^\times$-valued character $x \mapsto x^{\kappa_V + 1}$. On the other hand, $\cD_p^-$ has a basis vector which is $\Gamma$-invariant and on which $\varphi$ acts as $a_p(\cF)^{-1}$.

  So $ \cD_p^-$ is crystalline as a family of $(\varphi, \Gamma)$-modules; and $\cD_p^+$ is not crystalline as a family, but it is the twist of a crystalline family by the family of characters $\chi_{\mathrm{cyc}}^{(1 + \kappa_V)}$ where $\kappa_V$ is the universal character $\Zp^\times \to \cO(V)^\times$. Since integer powers of the cyclotomic character are crystalline, the specialisation $\cD_{p, k}^+$ of $\cD_p$ at any integer $k \in V \cap \ZZ$ is crystalline. More precisely, if $k \in V \cap \ZZ_{\ge 0}$ corresponds to a $p$-stabilisation of a classical level $N$ eigenform $f$, then $\Dcris(\cD_{p, k}^+)$ is canonically identified with the $\varphi =\frac{a_p(\cF)(k)}{p^{k+1}\veps(p)}$ eigenspace in $\Dcris( V_p(f)^*)$, where $f$ is the weight $k$ specialisation of $\cF$.

  \begin{rmk}
   Here we have identified $\Dcris( V_p(f)^*)$ with $\Dcris( V_p(f)^* \otimes \chi_{\mathrm{cyc}}^{-(1+k)})$. This identification depends on a choice of basis of $\Zp(1)$, i.e.~a compatible system of roots of unity $(\zeta_{p^n})_{n \ge 1}$; we choose $\zeta_{p^n}$ to be the image of $\exp(2 \pi i / p^n) \in \Qb \subset \CC$ under our embedding $\Qb \into \Qb_p$.
  \end{rmk}

  \begin{theorem}
   \label{thm:ES}
   There exists a canonical isomorphism of $\cO(V)$-modules
   \[ \omega_{\cF}': \Dcris\left(\cD_p^+ \otimes \chi_{\mathrm{cyc}}^{-(1 + \kappa_V)}\right) \xrightarrow{\ \sim\ } \cO(V), \]
   whose specialisation at any $k \in V \cap \ZZ_{\ge 0}$ coincides with the map
   \[ \Dcris( \cD_{p, k}^+) \into \Dcris(V_p(f)^*) \xrightarrow{\omega_f'} L,\]
   where $f$ is the weight $k$ specialisation of $\cF$, and $\omega_f' \in \operatorname{Fil}^1 H^1_{\mathrm{dR}, c}(Y_1(N), \cV_k)[f] \cong \Dcris(V_p(f))$ is the $L$-rational differential form associated to $f$ as in \cite{KLZ1a}.
  \end{theorem}

  This is a straightforward consequence of the overconvergent Eichler--Shimura isomorphism of Andreatta--Iovita--Stevens \cite{andreatta-iovita-stevens:overconvES}, combined with Liu's results; see \cite[Theorem 6.4.1]{loeffler-zerbes:coleman}.

  \begin{rmk}
   \label{rmk:gausssum}
   Note that the $q$-expansion of $\omega'_f$ at the cusp $\infty$ is $G(\veps^{-1})\cdot f(q)$; the Gauss sum is needed since the cusp $\infty$ is not $\QQ$-rational in our model of $Y_1(N)$. See \cite[Remark 6.1.4]{KLZ1a} for further discussion.
  \end{rmk}

 \subsection{Local conditions at \tp{p}}
  \label{sect:triangulations}

  Let $\fp \mid p$ be a prime of $\cK$. We let $\wD_{\fp} = \Drig\left(\wM |_{G_{\cK_{\fp}}}\right)$, which is a rank 2 $(\varphi, \Gamma)$-module over $\cR_{\cK_{\fp}} \htimes \cO(\tilde V)$, and we let $\wD_{\fp}^+$ be the rank 1 submodule given by Theorem \ref{thm:liu}.

  \begin{definition}
   By a \emph{classical point} of $\tilde V$, of weight $(a, b)$, we shall mean a point of $\tilde V$ corresponding to a $p$-stabilised Heegner pair $(f, \alpha, \chi)$, where $(f, \alpha)$ is a choice of $p$-stabilisation of some newform $f$ of level $N$ and weight $a + b + 2$, and $\chi$ is an algebraic Gr\"ossencharacter of $\infty$-type $(a, b)$. We write $\beta = p^{a + b + 1} \veps(p) / \alpha$ for the other root of the Hecke polynomial of $f$ at $p$.

   If $\chi$ has conductor prime to $p$, then we say $(f, \alpha, \chi)$ is a \emph{crystalline point}.
  \end{definition}

  \begin{proposition}
   \label{prop:bkselmer}
   Let $x = (f, \alpha, \chi)$ be a classical point, of weight $(a, b)$, with $a, b \ge 0$. Then the map $H^1(\cK_{\fp}, \wD_{\fp, x}^+) \to H^1(\cK_{\fp}, \wM_x)$ induced by \(\wD_{\fp}^+ \into \wD_{\fp}\) is an injection. Moreover, the Bloch--Kato subspaces $H^1_{\mathrm{e}}(\cK_{\fp}, \wM_x)$, $H^1_{\f}(\cK_{\fp}, \wM_x)$, and $H^1_{\mathrm{g}}(\cK_{\fp}, \wM_x)$ are all equal to the image of this map.
  \end{proposition}

  \begin{proof}
   Let us say that a de Rham $(\varphi, \Gamma)$-module $D$ over the Robba ring is ``generic'' if every subquotient $D'$ of $D$ satisfies $\Dcris(D)^{\varphi = 1} = \Dcris(D^*(1))^{\varphi = 1} = 0$. We claim that for a classical point $x$, the module $D = \wD_{p, x}$ is generic. If $\chi$ is non-crystalline at $p$, this is obvious; if $\chi$ is crystalline, then $D$ is crystalline and all eigenvalues of $\varphi$ on $\Dcris(D)$ and on $\Dcris(D^*(1))$ have complex absolute value $p^{-1/2}$, so none can be 1.

   The genericity of $D$ implies that $H^1_{\mathrm{e}}$, $H^1_{\f}$ and $H^1_{\mathrm{g}}$ coincide. Moreover, it also implies that $H^0(\cK_{\fp}, \wD_{\fp, x}^-) = 0$, since this module is a submodule of $\Dcris(\wD_{\fp, x}^-)^{\varphi = 1}$.

   The functor $\DD_{\mathrm{dR}}$, and the Bloch--Kato exponential, has been extended to $(\varphi, \Gamma)$-modules by Nakamura \cite{nakamura}; and if $D$ is generic and has all Hodge--Tate weights $\ge 1$ (resp.~$\le 0$) then the exponential map for $D$ is an isomorphism (resp.~is zero). Since $\wD_{\fp, x}^+$ has Hodge--Tate weights $\ge 1$, and $\wD_{\fp, x}^-$ has weights $\le 0$, the result follows.
  \end{proof}

  \begin{rmk}
   Note that the above reasoning would break down at a point corresponding to a classical newform of level $pN$ (although a closer study of the argument shows that the problem only arises if $a = b$ and $\chi$ is unramified at $\fp$). This is related to the phenomenon of \emph{exceptional zeros} of $p$-adic $L$-functions. However, our (rather strict) definition of ``classical point'' rules these out.
  \end{rmk}

  \begin{theorem}
   \label{thm:selmer2}
   The map $H^1(\cK_{\fp}, \wD_{\fp}^+) \to H^1(\cK_{\fp}, \wD_{\fp})$ induced by \(\wD_{\fp}^+ \into \wD_{\fp}\) is an injection, and $\operatorname{loc}_{\fp}(z_{\cF})$ is in the image of this map.
  \end{theorem}

  \begin{proof}
   As in \cite[Theorem 7.1.2]{loeffler-zerbes:coleman}, the fact that $\wD_{\fp}^-$ is a non-constant family of rank 1 modules implies that $H^0(\cK_{\fp}, \wD_{\fp}^-) = 0$. This gives the injectivity, and also implies that $H^1(\cK_{\fp}, \wD_{\fp}^-)$ is torsion-free. So it suffices to show that there is a Zariski-dense set of points $x \in \tilde{V}$ such that $z_{\cF}$ maps to 0 in $H^1(\cK_{\fp}, \wD_{\fp, x}^-)$.

   The specialisations of $\loc_\fp z_{\cF}$ at crystalline points $x$ of weight $(k-j, j)$, with $0 \le j \le k$, are in the image of the \'etale cycle class map; hence they lie in the Bloch--Kato $H^1_{\f}$ subspace (see e.g. \cite[\S 3.4]{BDP}). By the preceding proposition, they must map to 0 in $H^1(K_\fp, \wD_{\fp, x}^-)$. So $\loc_\fp z_{\cF}$ must vanish at all crystalline points satisfying this inequality, and these are clearly Zariski-dense.
  \end{proof}

 \subsection{The Selmer complex}

  We use here the formalism of Selmer complexes developed in \cite{pottharst13}. To apply Pottharst's theory to $\wM$, we need to make a choice of \emph{local conditions} $\Delta = (\Delta_v)_{v \in \Sigma}$; here $\Delta_v$ is the data of a perfect complex $\cO(\tilde V)$-modules $U_v^+$, and a map of complexes $\iota_v: U_v^+ \to R\Gamma(\cK_v, \wM)$. We write $U_v^-$ for the mapping cone of $\iota_v$.

  \begin{definition}
   Let $\Delta$ denote the following collection of local conditions:
   \begin{itemize}
    \item for $v \nmid p$, we choose $\Delta_v$ to be the unramified local condition $R\Gamma(\cK_v^{\mathrm{nr}}/\cK_v, \wM^{I_v})$;
    \item for $v \mid p$, we choose $\Delta_v$ to be the ``trianguline'' local condition
    \[ R\Gamma(\cK_v, \wD_v^+) \to R\Gamma(\cK_v, \wD_v).\]
   \end{itemize}
   We write $\widetilde{R\Gamma}(\cK, \wM; \Delta)$ for the corresponding Selmer complex.
  \end{definition}

  By construction there is an exact triangle
  \[ \widetilde{R\Gamma}(\cK, \wM; \Delta) \to R\Gamma(\OK[1/\Sigma], \wM) \to \bigoplus_{v \in \Sigma} U_v^- \to [+1]. \]
  It is clear that $H^0(\OK[1/\Sigma], \wM) = 0$ and that $H^0(U_v^-) = 0$ for $v\nmid p$; and we have seen above this is also true for $v \mid p$. Thus $H^0$ of the Selmer complex is zero, and its $H^1$ is given by
  \[ \widetilde{H}^1(\cK, \wM; \Delta) = \ker\left[ H^1\left(\OK[1/\Sigma], M_{\tilde V}(\cF)^* \right) \to \bigoplus_{v \in \Sigma} H^1(U_v^-)\right].
  \]
  Thus Proposition \ref{prop:selmer1} and Theorem \ref{thm:selmer2} can be summarized as stating that we have $z_{\cF} \in \widetilde{H}^1(\cK, \wM; \Delta)$.

  We finish this section by explaining the sense in which the Selmer complex ``interpolates'' the Bloch--Kato Selmer groups.

  \begin{lemma}
   Let $v \nmid p$ be a prime in $\Sigma$. Then $\wM^{I_v}$ is a direct summand of $\wM$ as a $\cO(\tilde V)$-module, and if $x$ is a crystalline point, we have $( \wM^{I_v})_x = (\wM_x)^{I_v}$.
  \end{lemma}

  \begin{proof}
   By assumption, $v$ is a degree 1 prime, and exactly one of $v$ and $\bar{v}$ divides $\fN$. We treat first the case $v \nmid \fN$.

   If $v \nmid \fN$, then the universal character $\kappa_{\tilde{V}}$ is unramified at $v$. So the $I_v$-invariants of $\wM = \cM \otimes_{\cO(V)} \cO(\tilde{V})(\kappa_V)$ are given by $\cM^{I_\ell} \otimes_{\cO(V)} \cO(\tilde V)$. Thus it suffices to show that $\cM^{I_\ell}$ is a direct summand of $\cM$ over $\cO(V)$, where $\ell$ is the rational prime below $v$, and we have $(\cM_k)^{I_\ell} = (\cM^{I_\ell})_k$ for all $k \in V \cap \ZZ_{\ge 0}$.

   Since $V$ is a 1-dimensional disc, $\cO(V)$ is a Dedekind domain. Thus  to check that $\cM^{I_\ell}$ is a direct summand, it suffices to show that $\cM / \cM^{I_\ell}$ is torsion-free. However, this is obvious: if $g \in \cO(V)$ is non-zero and $m \in \cM$ is such that $g \cdot m$ is $I_\ell$-invariant, then $g \cdot (im - m) = 0$ for all $i \in I_\ell$, and since $g \ne 0$ and $\cM$ is torsion-free, it follows that $im - m = 0$.

   Now let $k \in V \cap \ZZ_{\ge 0}$. For each $x \in V(\Qb_p)$ we have a Weil--Deligne representation $WD_\ell(\cM_x)$ associated to $\cM_x|_{G_{\QQ_\ell}}$, and the restriction of $WD_\ell(\cM_x)$ to the inertia subgroup of the Weil group is constant over $V$. So if  $(\cM_k)^{I_\ell}$ is strictly larger than $(\cM^{I_\ell})_k$, then the only possibility is that the action of inertia on the Weil--Deligne is trivial, but the monodromy operator $N$ is generically non-trivial but degenerates to 0 at $k$. In this case, the specialisation of $\cF$ at $k$ must be old, arising from a newform of level $N/\ell$; so the $N$-new component of the eigencurve on which $\cF$ lies must intersect with an old component at this point, which contradicts the smoothness of the eigencurve at non-critical classical points.

   We now briefly treat the case $v \mid \fN$. Let $\cF'$ denote the twisted Coleman family $\cF \otimes \veps^{-1}$. Then $\wM \cong M_V(\cF')^* \otimes_{\cO(V)} \cO_{\tilde{V}}(\kappa_{\tilde V} \cdot \veps^{-1})$, and $\kappa_{\tilde V} \cdot \veps^{-1}$ is unramified outside the primes dividing $\bar{\fN}$. So we may apply the same analysis to $\cF'$ instead of $\cF$ to obtain the result.
  \end{proof}

  \begin{corollary}
   Let $x = (f, \alpha, \chi)$ be a classical point of $\tilde V$. Then:
   \begin{enumerate}[(i)]
    \item We have
    \[
     \widetilde{R\Gamma}(\cK, \wM; \Delta) \otimes^{\mathbf{L}}_{\cO(\tilde V), x} \Qb_p \cong \widetilde{R\Gamma}(\cK, \wM_x; \Delta_x),
    \]
    where $\Delta_x$ denotes the local condition on $\wM_x = V_p(f)^*(\chi)$ defined by the unramified local conditions outside $p$ and the trianguline submodules $\wD_{v, x}$ at $v \mid p$.
    \item If $\chi$ has $\infty$-type $(a, b)$ with $a, b \ge 0$, then $\widetilde{H}^1(\cK, \wM_x; \Delta_x)$ is the Bloch--Kato Selmer group $H^1_{\f}(\cK, V_p(f)^*(\chi))$.
   \end{enumerate}
  \end{corollary}

  \begin{proof}
   The preceding lemma implies that the formation of the unramified local conditions is compatible with specialisation at $x$. This is true by definition for the trianguline conditions at $p$. So we obtain the compatibility of the Selmer complexes with specialisation from \cite[Theorem 1.12]{pottharst13}, proving (i).

   For (ii), we must check that the local conditions $\Delta_{x, v}$ agree with the Bloch--Kato local conditions at all primes $v$ of $\cK$. Away from $p$ this is true by definition; for $v \mid p$ it is part of Proposition \ref{prop:bkselmer}.
  \end{proof}

  \begin{rmk}
   For a single modular form $f$ (rather than a family), assumed to be ordinary at $p$, Perrin-Riou has formulated an ``anticyclotomic Iwasawa main conjecture without $p$-adic $L$-functions'' \cite[Conjecture B]{perrinriou87}; this conjecture been proved, under some mild technical hypotheses, in \cite{BCK}. In the language of Selmer complexes, this conjecture relates the index of a Heegner class in $H^1$ of the anticyclotomic Selmer complex to the torsion submodule of $H^2$ of the same complex. It would be interesting to formulate a generalisation of Perrin-Riou's conjecture in the present setting; it may even be possible to prove some results in this direction using the ``Euler system machine'' to bound Selmer groups, analogous to \cite{arizonafolks} in the Rankin--Selberg case. We shall not pursue this matter in the present paper, but it would be a very interesting direction for further work.
  \end{rmk}

%%%%%%%%%%%%%%%%%%%%%%%%%%%%%%%%%%

\section{P-adic \tp{L}-functions for split \tp{p}}

 We now assume that $p$ is split in $\cK$; and we write $(p) = \fp \fpb$, where $\fp$ is the prime corresponding to our embedding $\cK \subset \Qb \into \Qb_p$. In order to make use of the local computations of \cite{BDP}, we shall also suppose, as in \emph{op.cit.}, that $\cK$ has odd discriminant.
%
% Note that the Galois character $\bar\sigma: \Gal(\cK^{\ab} / F) \to \Qp^\times$ is unramified at the prime $\fP$ of $F$ above $\fp$ determined by our embeddings; and if $h$ denotes the order of $\fp$ in the ray class group modulo $\fN$, then we have
% \[ \bar\sigma(\operatorname{Frob}_{\fP}^{-1}) = \bar\sigma(\nu), \]
% where $\nu$ is the unique generator of $\fp^h$ that is 1 modulo $\fN$.

 \subsection{CM periods}
  \label{sect:CMperiods}

  Recall the pair $(A, t)$ chosen above, with $A$ isomorphic to $\CC / \OK$ as an elliptic curve over $\CC$. Let $\Qpnh \subset \CC_p$ denote the completion of the maximal unramified extension of $\Qp$, and $\Zpnh \subset \cO_{\CC_p}$ denote its ring of integers (equivalently, we can define $\Zpnh$ as the ring of Witt vectors of $\overline{\mathbf{F}}_p$).

  \begin{definition}
   Let $\omega_A$ be a choice of basis of the 1-dimensional $F$-vector space $\Omega^1(A / F)$. We define the following two periods:
   \begin{itemize}
    \item Let $\Omega_\infty \in \CC$ be such that
    \[ \omega_A = \Omega_\infty \cdot 2\pi i\, \mathrm{d}w \]
    as bases of $\Omega^1(A / \CC)$, where $w$ is the standard complex coordinate on $A(\CC) \cong \CC /\OK$ (cf.~\cite[Eq. (5.1.16)]{BDP}).
    \item Let $\Omega_{\fp} \in \Zpnh$ be such that
    \[ \omega_A = \Omega_{\fp} \cdot \omega_{\mathrm{can}}, \]
    as bases of $\Omega^1(A[\fp^\infty] / \Zpnh)$, where $\omega_{\mathrm{can}}$ is the pullback to $A$ of the differential $\frac{\mathrm{d}u}{u}$ on $\GG_m$, via an isomorphism of $p$-divisible groups $A[\fp^\infty] \cong \mu_{p^\infty}$ over $\Zpnh$ (cf.~\cite[Eq. (5.2.2)]{BDP}).
   \end{itemize}
  \end{definition}

  In the language of $p$-adic Hodge theory, $\omega_A$ is a basis of $\Dcris(\sigma|_{G_{F_{\fP}}}) = \left( \mathbf{B}_{\cris}\right)^{G_{F_{\fP}} = \sigma^{-1}}$, and we have $\omega = \Omega_{\fp} \cdot t$, where $t$ is the period for the cyclotomic character. In particular, this shows that $G_{F_{\fP}}$ acts on $\Omega_{\fp}$ via the character $\sigma^{-1}\cdot \chi_{\mathrm{cyc}}^{-1} = \bar\sigma$.

  We let $\eta_A$ be the basis of $H^1_{\mathrm{dR}}(A / F_{\fP}) / \mathrm{Fil}^1$ which pairs to 1 with $\omega$. As an element of $\mathbf{B}_{\cris}$ this is simply $1/\Omega_{\fp}$.

 \subsection{The BDP \tp{p}-adic \tp{L}-function of a cusp form} We now recall the two main results of \cite{BDP}: firstly, the construction of a 1-variable $p$-adic $L$-function interpolating square roots of central $L$-values $L(f, \chi^{-1}, 1)$, for a fixed $f$ and varying $\chi$; secondly, its relation to the Heegner classes for $f$.

  Let $f$ be a newform of level $N$, character $\veps$ and weight $k + 2 \ge 2$. Write $\cW_{\cK}(\fN,k,\veps)$ for the fibre of $\cW_{\cK}(\fN, \veps)$ above $k \in \cW$, parametrising characters $\chi$ of $\Gal(\cK^{\ab} /\cK)$ unramified outside $\fN p^\infty$ such that $V_p(f)^*\otimes \chi$ is conjugate self-dual. We divide the set of classical points $\chi$ of $\cW_{\cK}(\fN, k, \veps)$ into three subsets $\Sigma^{(1)} \sqcup \Sigma^{(2)} \sqcup \Sigma^{(2')}$, as in Definition 4.1 of  \cite{BDP}, according to the $\infty$-type $(a,b)$:
  \begin{itemize}
   \item $\chi \in \Sigma^{(1)}$ if $0 \le a,b \le k$,
   \item $\chi \in \Sigma^{(2)}$ if $a \ge k+1$ and $b \le -1$,
   \item $\chi \in \Sigma^{(2')}$ if $a \le -1$ and $b \ge k+1$.
  \end{itemize}
  The $\chi \in \Sigma^{(1)}$ are precisely those for which $(f, \chi)$ is a Heegner pair. On the other hand, if $\chi \in \Sigma^{(2)}$ then the sign in the functional equation of $L(f, \chi^{-1}, s)$ is $+1$, so the central value $L(f, \chi^{-1}, 1)$ is not forced to vanish.

  \begin{notation}
   For $\chi \in \Sigma^{(2)}$ of $\infty$-type $(a, b)$, let us write
   \[ L^{\mathrm{alg}}(f, \chi^{-1}, 1) = \frac{C(f, \chi^{-1}, 1)}{w(f, \chi) \cdot \Omega_\infty^{2(a-b)}} L(f, \chi^{-1}, 1) \in \Qb, \]
   as in \S 5.1 of \cite{BDP}, where $C(f, \chi^{-1}, 1)$ and $w(f, \chi)$ are appropriate modifying factors as in \emph{op.cit.}.
  \end{notation}

  \begin{theorem}[{\cite[Theorem 5.5]{BDP}}]
   There exists a bounded rigid-analytic function $\BDP(f)$ on $\cW_{\cK}(\fN, k, \veps)$, whose value at a crystalline character $\chi \in \Sigma^{(2)}$ of $\infty$-type $(a, b)$ is given by
   \[
    \frac{\BDP(f)(\chi)}{\Omega_{\fp}^{(a-b)}}
    = (1 - \tfrac{\chi(\fp)}{\alpha}) (1 - \tfrac{\chi(\fp)}{\beta})
    \left[ L^{\mathrm{alg}}(f, \chi^{-1}, 1) \right]^{1/2}.\qedhere\qed
   \]
  \end{theorem}

  \begin{rmk}
   We apologise for some clashes of notation between the present paper and \emph{op.cit.}: as noted above, their $\chi$ is $\chi \cdot \mathbf{N}$ in our notation, where $\mathbf{N}$ is the norm character of $\infty$-type $(1, 1)$, and their $k$ is our $k + 2$. It follows from work of Brakocevic \cite{brakocevic11} that $\BDP(f)$ also enjoys an interpolating property at non-crystalline characters in $\Sigma^{(2)}$ (ramified at the primes above $p$), but we shall not make this explicit here, since the crystalline characters in $\Sigma^{(2)}$ are already dense in $\cW(\fN, k, \veps)$, and so the interpolating property at these points is already sufficient to uniquely determine $\BDP(f)$.
  \end{rmk}

  \begin{theorem}[{\cite[Theorem 5.13]{BDP}}]
   \label{thm:explrecip}
   The value of $\BDP(f)$ at a crystalline character in $\Sigma^{(1)}$, of $\infty$-type $(a,b)$ with $a,b \ge 0$, is given by
   \[
   \frac{\BDP(f)}{\Omega_{\fp}^{(a-b)}} = \frac{1}{b!\binom{a + b}{a}G(\veps^{-1})} \left\langle \log\left(\loc_\fp z_{\et}^{[f, \chi]}\right), \omega_f' \wedge \omega_A^b \eta_A^{a}\right\rangle,
   \]
   where $\omega_A$ and $\eta_A$ are differentials on the CM elliptic curve $A = \CC/ \OK$ as in \S\ref{sect:CMperiods}, and $G(\veps^{-1})$ is the Gauss sum.
  \end{theorem}

  \begin{proof}
   This is Theorem 5.13 of \emph{op.cit.} translated into our conventions. The factor $\binom{a+b}{a}$, which does not appear in \emph{op.cit.}, reflects our use of symmetric tensors $\TSym^k$, rather than the symmetric powers $\Sym^k$, in our definition of the Heegner class. For the factor $G(\veps^{-1})$ see Remark \ref{rmk:gausssum} above.
  \end{proof}

 \subsection{BDP \tp{p}-adic \tp{L}-functions in families}

  Our next result shows that the $p$-adic $L$-function $\BDP(f)$ can be extended to a rigid-analytic function on the 2-dimensional space $\cE_{\cK}(\fN, \veps)$, allowing the form $f$ to vary as well as the character $\chi$. Note that this extends a result for ordinary families due to Castella \cite[Theorem 2.11]{castella:padicvar}.

  Recall that the BDP $p$-adic $L$-function satisfies the following formula, for any crystalline $\chi$ of $\infty$-type $(a,b)$:
  \begin{equation}
   \label{eq:defBDP}
   \BDP(f)(\chi) = \sum_{[\mathfrak{c}] \in \operatorname{Cl}(K)} \chi(\mathfrak{c})^{-1}\mathbf{N}(\mathfrak{c})^{b} \theta^{-(1+b)} f^\flat\left( \mathfrak{c} \star (A, t, \omega_{\mathrm{can}})\right).
  \end{equation}
  Here $f^\flat$ denotes the ``$p$-depletion'' of $f$, i.e.~the $p$-adic modular form with $q$-expansion $\sum_{p \nmid n} a_n(f) q^n$; and $\theta$ is the differential operator $q \tfrac{\mathrm{d}}{\mathrm{d}q}$ on $p$-adic modular forms. For the notation $\mathfrak{a} \star (A, t, \omega_{\mathrm{can}})$, see Equation 1.4.8 of \emph{op.~cit.}; note that if $\mathfrak{c} = (\xi)$ is a principal ideal, with $\xi \in \OK$ coprime to $\fN$, then we have
  \[
   \mathfrak{c} \star (A, t, \omega_{\mathrm{can}}) = (A, \xi \cdot t, \xi^{-1}\omega_{\mathrm{can}}).
  \]

  We now recall the notion of a \emph{family of $p$-adic modular forms} over $X$, where $X$ is a rigid space with a map $\kappa_X: X \to \cW$ (see \cite{coleman97} for further details). The construction of the eigencurve gives rise to a universal family of eigenforms $\cF_{\mathrm{univ}}$ over $\cE(N, \veps)$. We may therefore consider its $p$-depletion $\cF_{\mathrm{univ}}^\flat$, which is a family of infinite-slope eigenforms.

  We want to ``twist'' this by appropriately-chosen characters. We identify $\cO_{K, p}^\times$ with $\Zp^\times \times \Zp^\times$ via the two primes $(\fp, \fpb)$. Mapping $x \in \Zp^\times$ to either $(x^{-1}, 1)$ or $(1, x^{-1})$ defines maps $\Zp^\times \into \cO_{K, p}^\times \subset \AA_{\cK, \f}^\times$, and hence two characters $\Zp^\times \to \cO(\cW_{\cK}(\fN, \veps))^\times$, which we denote by $\mathbf{a}$ and $\mathbf{b}$ (since their specialisations at a crystalline character of $\infty$-type $(a,b)$ are the integers $a$ and $b$).% Observe that $\mathbf{a}+ \mathbf{b}$ is the pullback to $\cW_{\cK}(\fN, \veps)$ of the universal character $\Zp^\times \to \cO(\cW)^\times$. Then $\theta^{-1-\mathbf{b}}(f^\flat)$ is a family of $p$-adic modular forms over $\cW_{\cK}$, and

  \begin{proposition}
   There exists a family of $p$-adic modular forms $\theta^{-(1+\mathbf{b})}\cF_{\mathrm{univ}}^\flat$ over $\cE_{\cK}(\fN, \veps)$, whose specialisation at a crystalline point $(f, \alpha, \chi)$ with $\chi$ of weight $(a,b)$ is $\theta^{-(1+b)}(f^\flat)$.
  \end{proposition}

  \begin{proof}
   Since we know that $\cF_{\mathrm{univ}}$ and hence $\cF_{\mathrm{univ}}^\flat$ exist as families of $p$-adic modular forms, it suffices to note the following general statement: if $X$ is a rigid space with two maps $\kappa, \lambda: X \to \cW$, and $\mathcal{G}$ is a family of $p$-adic modular forms over $X$ of weight $\kappa$ such that $U_p \cdot \mathcal{G} = 0$, then there exists a $p$-adic modular form $\theta^\lambda(\mathcal{G})$ of weight $\kappa + 2\lambda$ whose $q$-expansion is $\sum_{p \nmid n} a_n(\mathcal{G}) n^\lambda$. See e.g.~\cite[\S 2.6]{darmonrotger14}, or the introduction to \S 4 of \cite{ai-tripleprod}.
  \end{proof}

  \begin{rmk}
   Note that although $\cF_{\mathrm{univ}}$ and hence $\cF_{\mathrm{univ}}^\flat$ are families of overconvergent forms, $\theta^{-(1+\mathbf{b})}\cF_{\mathrm{univ}}^\flat$ is not overconvergent. For weights in a certain open subset in the ``centre'' of weight space, it can be interpreted as a family of \emph{nearly-overconvergent} forms in the sense of \cite{ai-tripleprod}. However, this is not needed for our present purposes.
  \end{rmk}

  \begin{theorem}
   \label{thm:B2}
   There is a rigid-analytic function $\BDP(\fN, \veps)$ on $\cE_{\cK}(\fN, \veps)$ with the following property: if $f_\alpha$ is a point on $\cE(N, \veps)$ corresponding to a $p$-stabilisation of a newform $f$ of prime-to-$p$ level and weight $k+2$, then the restriction of $\BDP(\fN, \veps)$ to the fibre of $\cE_{\cK}(\fN, \veps)$ over $f_\alpha$ is $\BDP(f)$.
  \end{theorem}

  \begin{proof}
   We simply define $\BDP(\fN, \veps)$ to be the natural ``family version'' of \eqref{eq:defBDP}, i.e.~the sum
   \[\sum_{[\mathfrak{c}] \in \operatorname{Cl}(K)} \kappa_{\mathrm{univ}}(\mathfrak{c})^{-1}\mathbf{N}(\mathfrak{c})^{\mathbf{b}} \cdot  \theta^{-(1+\mathbf{b})}(\cF_{\mathrm{univ}}^\flat)\left( \mathfrak{c} \star (A, t, \omega_{\mathrm{can}})\right),
   \]
   where $\kappa_{\mathrm{univ}}: \AA_{\cK, \f}^\times/\cK^\times \to \cO(\cW_{\cK}(\fN, \veps))$ is the universal character. By construction, restricting this to the fibre above $f_\alpha$ gives back the BDP $p$-adic $L$-function above.
  \end{proof}

  \begin{rmk}
   This is in a sense the `wrong' setting for variation of BDP $L$-functions, for the following reason: a level $N$ newform $f$ of weight $\ge 2$ and character $\veps$ will have two $p$-stabilisations $f_\alpha$ and $f_\beta$, which are usually (and conjecturally always) distinct as points in $\cE(N, \veps)$; but the restrictions of $\BDP(\fN, \veps)$ to the fibres of $\cE_{\cK}(\fN, \veps)$ above these two points are the same, since both are equal to $\BDP(f)$, and this function is insensitive to choices of $p$-stabilisations.

   In a recent preprint of one of us \cite{DL-deformations}, assuming $N = 1$ and imposing some mild local assumptions on the mod $p$ Galois representation $\bar\rho$ of $f$, it is shown that $\BDP(f)$ extends to the product of $\cW^\ac$ and a 2-dimensional Galois deformation space $\mathcal{X}$. There is a natural map from the eigencurve $\cE$ to $\mathcal{X}$ (whose image is the ``infinite fern'' of Mazur), and the function $\BDP(\fN, \veps)$ above is simply the pullback of the ``universal'' BDP $L$-function on $\mathcal{X} \times \cW^\ac$ along this map. Note that $f_\alpha$ and $f_\beta$ map to the same point of $\mathcal{X}$, which explains the phenomenon mentioned above. However, for the constructions of the next section, the additional data carried by the eigencurve -- a choice of $p$-stabilisation, or equivalently a triangulation of the Galois representation -- is indispensable.
  \end{rmk}

\section{Explicit reciprocity laws}

  \label{sect:ERL}

 \subsection{Period isomorphisms in families}

  We now study the local triangulation $\wD_{\fp}$ of \S \ref{sect:triangulations} in the split-prime case, assuming $\fp$ is the prime corresponding to our embedding $\cK\into \Qb_p$ as before.

  If $\chi_{\mathrm{cyc}}$ denotes the cyclotomic character, then $\wD_{\fp}^+ \otimes \chi_{\mathrm{cyc}}^{-(1 + \mathbf{b})}$ is an unramified, and in particular crystalline, $(\varphi, \Gamma)$-module; we shall write
  \[ \wM^+_{\cris} \coloneqq \Dcris\left(\wD_{\fp}^+ \otimes \chi_{\mathrm{cyc}}^{-(1 + \mathbf{b})}\right),\]
  which is a locally-free rank 1 $\cO(\tilde V)$-module. By construction, for any crystalline point $(f, \alpha, \chi)$ of weight $(a, b)$, there is a canonical inclusion
  \[ (\wM^+_{\cris})_x \into \Dcris(\wM_x|_{G_{\cK_{\fp}}}) = \Dcris(V_p(f)^*(\chi) |_{G_{\cK_{\fp}}}). \]
  We can regard $\wM^+_{\cris}$ as the tensor product of two factors:
  \begin{itemize}
   \item the pullback to $\tilde{V}$ of the $(\varphi, \Gamma)$-module $\cD_p^+ \otimes \chi_{\mathrm{cyc}}^{-(\kappa_V+ 1)}$ over $V$;
   \item the $\cO(\tilde V)$-valued character $\kappa_{\tilde V} \cdot (\chi_{\mathrm{cyc}})^{\mathbf{a}}$, which is the restriction to the decomposition group at $\fp$ of a character of $\Gal(\cK^{\ab} / \cK)$ unramified outside $\fN\fpb$.
  \end{itemize}

  Theorem \ref{thm:ES} allows us to identify $\Dcris\left(\cD_p^+ \otimes\chi_{\mathrm{cyc}}^{-(\kappa_V+ 1)}\right)$ with $\cO(V)$, via the linear map $\omega_\cF'$. By base-extension to $\cO(\tilde V)$ we may consider this as an isomorphism $\Dcris\left(\wD_p^+ \otimes\chi_{\mathrm{cyc}}^{-(\mathbf{a} + \mathbf{b} + 1)}\right)\cong\cO(\tilde V)$. As for the second factor, since $\kappa_{\tilde V} \cdot (\chi_{\mathrm{cyc}})^{\mathbf{a}}$ is unramified at $\fp$, we have a canonical inclusion
  \[
   \Dcris\left(\kappa_{\tilde V} \cdot (\chi_{\mathrm{cyc}})^{\mathbf{a}}\right) \subset \cO(\tilde V) \htimes \Qpnh
  \]
  (where both sides are regarded as subsets of $\mathbf{B}_{\cris} \htimes \cO(\tilde V)$). Tensoring together these two constructions gives a \emph{period isomorphism}
  \begin{equation}
   \label{eq:ESp}
   \widetilde{\omega}_{\cF}:  \wM^+_{\cris} \xrightarrow{\ \ }\cO(\tilde V) \htimes \Zpnh.
  \end{equation}

  \begin{proposition}
   Let $x = (f, \alpha, \chi)$ be a classical crystalline point, with $\chi$ having $\infty$-type $(a,b)$. Then the specialisation $\widetilde{\omega}_{\cF, x}$ of $\widetilde{\omega}_{\cF}$ at $x$ takes values in $F_{\fP} \cdot \Omega_{\fp}^{(a-b)} \subset \Qpnh$, and the morphism
   \[ \frac{1}{\Omega_{\fp}^{(a-b)}} \widetilde{\omega}_{\cF, x}: \wM^+_{\cris, x} \to F_{\fP} \]
   is the restriction to $\wM^+_{\cris, x}$ of the map $\Dcris\left(V_p(f)^*(\chi) |_{G_{\cK_{\fp}}}\right) \to F_{\fP}$ given by $x \mapsto \langle x, \omega_f' \otimes \eta_A^a \omega_A^b \rangle$.
  \end{proposition}

  \begin{proof}
   This follows readily by comparing the construction of $\widetilde{\omega}_{\cF}$ with the definition of $\Omega_{\fp}$ in the previous section.
  \end{proof}

 \subsection{The Perrin-Riou regulator at \texorpdfstring{$\fp$}{p}}

  \begin{proposition}
   There is a canonical homomorphism of $\cO(\tilde V)$-modules, the \emph{Perrin-Riou regulator},
   \[ \cL_{\fp, \wM}: H^1\left(\cK_{\fp}, \wD_{\fp}^+\right) \to  \wM^+_{\cris}, \]
  such that for any crystalline point $x = (f, \alpha, \chi)$ of $\tilde V$, there is a commutative diagram
   \[
    \begin{tikzcd}
     H^1(\cK_{\fp}, \wD^+_{\fp, x}) \dar[hook] \rar["\cL_{\fp, \wM, x}"] &\wM^+_{\cris, x}\dar[hook] \\
     H^1(\cK_{\fp}, \wM_x) \rar["\cE_\fp \cdot \spadesuit"] &\Dcris(\wM_x |_{G_{\cK_{\fp}}})
    \end{tikzcd}
   \]
   where $\cE_{\fp}$ is the Euler factor $\left(1 - \tfrac{\alpha}{p \chi(\fpb)}\right) \left(1 - \tfrac{\chi(\fpb)}{\alpha}\right)^{-1}$, and the map $\spadesuit$ is as follows:
   \begin{itemize}
    \item If $\chi$ has weight $(a, b)$ with $b \le -1$, then $\spadesuit = (-1-b)!  \exp^*$, where
    \[ \exp^*:  \frac{H^1(\cK_{\fp}, \wM_x)}{H^1_{\f}(\cK_{\fp}, \wM_x)} \xrightarrow{\ \cong\ } \operatorname{Fil}^0 \Dcris(\wM_x |_{G_{\cK_{\fp}}})  \]
    is the Bloch--Kato dual exponential map.

    \item If $b \ge 0$, then the image of the left vertical map is contained in $H^1_{\f}(\cK_{\fp},  \wM_x)$, and on this submodule $\spadesuit$ satisfies the relation
    \[  \spadesuit \bmod {\operatorname{Fil}^0} = \tfrac{(-1)^b}{b!} \log, \]
    where
    \[ \log: H^1_{\f}(\cK_{\fp}, \wM_x) \xrightarrow{\ \cong\ } \frac{\Dcris(\wM_x |_{G_{\cK_{\fp}}})}{\operatorname{Fil}^0}\]
    is the Bloch--Kato logarithm.
   \end{itemize}
  \end{proposition}

  \begin{proof}
   Exactly as in the case of Rankin--Selberg convolutions treated in \cite[\S 6--7]{loeffler-zerbes:coleman}, the map $\cL_{\fp, \wM}$ is defined simply to be the composite of $(1 - \varphi)$ on the $(\varphi, \Gamma)$-module with the Mellin transform, identifying the kernel of $\psi$ with $\Dcris$-valued distributions. The content of the theorem is that this map is compatible under specialisation with the classical logarithm and exponential maps; this follows from Nakamura's construction of $\exp^*$ and $\log$ for $(\varphi, \Gamma)$-modules \cite{nakamura}.
  \end{proof}

  \begin{rmk}
   Note that the Euler factor $\cE_{\fp}$ is well-defined and non-zero, since the genericity property established in Proposition \ref{prop:bkselmer} shows that $\left(1 - \tfrac{\alpha}{p \chi(\fpb)}\right)$ and $\left(1 - \tfrac{\chi(\fpb)}{\alpha}\right)$ are both non-vanishing. This genericity property also implies that $H^1_{\mathrm{e}} = H^1_{\f}$, so the Bloch--Kato logarithm is defined on $H^1_{\f}$. The interpolating property at non-crystalline classical points can be made explicit, but we do not need this here.
  \end{rmk}

  \begin{corollary}
   The element
   \[ \cL_{\fp, \mathrm{mot}}(\cF) \coloneqq \frac{(-1)^{\mathbf{b}}}{G(\veps^{-1})} \left\langle\cL_{\fp, \wM}\left(\loc_\fp z_{\cF}\right), \widetilde{\omega}_{\cF}  \right\rangle \in \cO(\tilde V) \]
   satisfies
   \[ \frac{1}{\Omega_{\fp}^{(a-b)}} \cL_{\fp, \mathrm{mot}}(\cF)(x) = \frac{a!}{(a + b)! G(\veps^{-1})} \left(1 - \tfrac{\chi(\fp)}{\alpha}\right) \left(1 - \tfrac{\chi(\fp)}{\beta}\right) \left\langle \log(\loc_\fp z_{\et}^{[f ,\chi]}), \omega_f' \wedge \omega_A^a \eta_A^b\right\rangle \]
   for any classical crystalline point $x = (f, \alpha, \chi)$ with $\chi$ of weight $(a, b)$ such that $a, b \ge 0$.
  \end{corollary}

  \begin{proof}
   This follows by comparing the interpolation property of the \'etale class $z_{\cF}$ from Theorem A with the interpolation property of the map $\cL_{\fp, \wM}$: two of the four Euler factors cancel out, and we obtain the result stated.
  \end{proof}

  \begin{theorem}[Theorem B]
   The motivic $p$-adic $L$-function $\cL_{\fp, \mathrm{mot}}(z_{\cF})$ coincides with the restriction to $\tilde V \subset \cE_\cK(\fN, \veps)$ of the $p$-adic $L$-function $\BDP(\fN, \veps)$ of Theorem \ref{thm:B2}.
  \end{theorem}

  \begin{proof}
   Let $x = (f, \alpha, \chi)$ be a classical crystalline point of $\infty$-type $(a, b)$, with $a, b \ge 0$. By Theorem \ref{thm:B2}, the value of $\BDP(\fN, \veps)$ at $x$ is $\BDP(f)(\chi)$; and by Theorem \ref{thm:explrecip}, this is equal to
   \[  \Omega_{\fp}^{(a-b)} \cdot \frac{a!}{(a + b)! G(\veps^{-1})} \left(1 - \tfrac{\chi(\fp)}{\delta}\right) \left(1 - \tfrac{\chi(\fp)}{\beta}\right) \left\langle \log(\loc_\fp z_{\et}^{[f ,\chi]}), \omega_f' \wedge \omega_A^a \eta_A^b\right\rangle.\]
   This is exactly the formula we have just obtained for $\cL_{\fp, \mathrm{mot}}(\cF)(x)$. Since crystalline points with $a, b \ge 0$ are Zariski-dense in $\tilde V$, the ``motivic'' $p$-adic $L$-function must be equal to the restriction of the ``analytic'' one $\BDP(\fN, \veps)$.
  \end{proof}
\providecommand{\bysame}{\leavevmode\hbox to3em{\hrulefill}\thinspace}
\providecommand{\MR}[1]{}
\renewcommand{\MR}[1]{%
 MR \href{http://www.ams.org/mathscinet-getitem?mr=#1}{#1}.
}
\providecommand{\href}[2]{#2}
\newcommand{\articlehref}[2]{\href{#1}{#2}}

\end{document}